\documentclass[journal]{IEEEtran}
\usepackage{cite}
\usepackage[pdftex]{graphicx}
\usepackage{algorithmic}
\usepackage{array}
\ifCLASSOPTIONcompsoc
 \usepackage[caption=false,font=normalsize,labelfont=sf,textfont=sf]{subfig}
\else
 \usepackage[caption=false,font=footnotesize]{subfig}
\fi
\usepackage{stfloats}
\ifCLASSOPTIONcaptionsoff
 \usepackage[nomarkers]{endfloat}
\let\MYoriglatexcaption\caption
\renewcommand{\caption}[2][\relax]{\MYoriglatexcaption[#2]{#2}}
\fi
\usepackage{url}
\usepackage{booktabs} 
\usepackage{amsmath,lmodern}
\usepackage{bibentry}
\usepackage{amsthm}
\usepackage{amssymb}
\usepackage{pifont}
\usepackage{mathtools}
\usepackage{amsmath,mleftright}
\usepackage{xparse}
\usepackage{makecell}

\NewDocumentCommand{\evalat}{sO{\big}mm}{%
  \IfBooleanTF{#1}
   {\mleft. #3 \mright|_{#4}}
   {#3#2|_{#4}}%
}

\newcommand{\comment}[1]{#1}

\newcommand\mynorm[1]{\left\lVert#1\right\rVert}
\newtheorem{theorem}{Theorem}
\newtheorem{lemma}{Lemma}
\newtheorem{assumption}{Assumption}
\newtheorem{definition}{Definition}

\newtheorem{example}{Example}

\newtheorem{remark}{Remark}

\newenvironment{ealign}
    {\begin{equation}
    \begin{aligned}
    }
    { 
    \end{aligned} 
    \end{equation}
    }

\begin{document}
%
\title{Time-variation in online nonconvex optimization enables escaping from spurious local minima}
%
%
%

\author{Yuhao Ding,
        Javad Lavaei, and Murat Arcak
\thanks{Y. Ding and J. Lavaei are with the Department of Industrial Engineering and Operations Research, University of California, Berkeley, CA 94709 USA (e-mail: yuhao\_ding@berkeley.edu; lavaei@berkeley.edu).}
\thanks{M. Arcak is with the Department of Electrical Engineering and Computer Sciences, University of California, Berkeley, CA 94709 USA (e-mail: arcak@berkeley
edu).}
\thanks{\comment{A shorter version of this paper has been submitted for the conference publication \cite{ding2019escaping}.
The new additions to this version include the analysis of optimization problems with equality constraints (as opposed to unconstrained optimization in the conference version), the derivation of the time-varying projected gradient flow system, the role of time variation of the constraints, a unified view of the analysis of equality-constrained problems and  unconstrained problems by introducing the Langrange functional, and the time-varying analysis for jumping behavior.}}
}

%
%

%

\maketitle

\begin{abstract}
A major limitation of online algorithms that track the optimizers of time-varying nonconvex optimization problems is that they focus on a specific local minimum trajectory, which may lead to poor spurious local solutions. In this paper, we show that the natural temporal variation may help simple online tracking methods find and track time-varying global minima. To this end, we investigate the properties of a time-varying projected gradient flow system with inertia, which can be regarded as the continuous-time limit of (1) the optimality conditions for a discretized sequential optimization problem with a proximal regularization and (2) the online tracking scheme. 
\comment{We introduce the notion of the dominant trajectory and show that the inherent temporal variation could re-shape the landscape of the Lagrange functional and help a proximal algorithm escape the spurious local minimum trajectories if the global minimum trajectory is dominant. For a problem with twice continuously differentiable objective function
and constraints,}
sufficient conditions are derived to guarantee that no matter how a local search method is initialized, it will track a time-varying global solution after some time. The results are illustrated on a benchmark example with many local minima.

\end{abstract}


%
\IEEEpeerreviewmaketitle
\section{Introduction}
%
%
%
%
\IEEEPARstart{I}{n} this paper, we study the following equality-constrained time-varying optimization problem:
\begin{ealign} \label{eq: constrained time-varying optimization}
\min_{x(t)\in \mathbb{R}^n} &\quad f(x(t),t)\\
\text{s.t.} &\quad g(x(t),t)=0
\end{ealign}

\noindent where $t\geq 0$ denotes the time and  $x(t)$ is the optimization variable that depends on $t$. 
Moreover, the objective function $\comment{f}: \mathbb{R}^n\times [0,\infty)\rightarrow \mathbb{R}$ and the constraint function $g(x,t)=(g_1(x,t),\ldots,g_m(x,t))$ with $\comment{g_k}: \mathbb{R}^n\times [0,\infty)\rightarrow \mathbb{R}$ for $k=1,...,m$ are assumed to be \comment{twice continuously differentiable in state $x$ and continuously differentiable in time $t$.}
For each  time $t$, the function $f(x,t)$ could potentially be nonconvex  in $x$ with many local minima and the function $g(x,t)$ could also potentially be nonlinear in $x$, leading to a nonconvex feasible set. The objective is to solve the above problem online under the assumption that at any given time $t$ the function $f(x,t^\prime)$ and $g(x,t^\prime)$ are known for all $t^\prime\leq t$ while no knowledge about $f(x,t^\prime)$ or $g(x,t^\prime)$ may be available for any $t^\prime>t$. Therefore, the problem \eqref{eq: constrained time-varying optimization} cannot be minimized off-line and should be solved sequentially. Another
issue is that the optimization problem at each time instance could be highly complex due to NP-hardness, which
is an impediment to finding its global minima.
 This paper aims to investigate under what conditions simple local search algorithms can  solve the above online optimization problem to almost global optimality after some finite time.  More precisely, the goal is  to devise an algorithm that can track a global solution of \eqref{eq: constrained time-varying optimization} as a function of time $t$ with some error at the initial time and a diminishing error after some time.

If $f(x,t)$ and $g(x,t)$ do not change over time, the problem reduces to a classic (time-invariant) optimization problem. It is known that  simple  local search methods, such as stochastic gradient descent (SGD) \cite{hazan2016introduction}, may be able to find a global minimum of such time-invariant problems (under certain conditions) for almost all initializations due to the randomness embedded in SGD \cite{jin2017escape,ge2015escaping,kleinberg2018alternative}. The objective of this paper is to significantly extend the above result from a single optimization problem to infinitely-many problems parametrized by  time $t$. In other words, it is desirable to investigate the following question:  \textbf{Can the temporal variation  in the landscape of  time-varying nonconvex optimization problems enable online local search methods to find and track global trajectories?} To answer this question,  we study a first-order time-varying ordinary differential equation (ODE), which is the counterpart of the classic projected gradient flow system for time-invariant optimization problems \cite{tanabe1974algorithm} and serves as a continuous-time limit of the discrete online tracking method for \eqref{eq: constrained time-varying optimization} with the proximal regularization. This ODE is given as
\begin{equation} \label{eq: ODE3} \tag{P-ODE}
\dot{x}(t) = -\frac{1}{\alpha}\mathcal{P}(x(t),t) \nabla_x f(x(t),t) - \mathcal{Q}(x(t),t)g^\prime(x(t),t)
\end{equation}
where $\alpha>0$ is a constant parameter named \textbf{inertia} due to a \textbf{proximal regularization}, \comment{ $g^\prime(z,t)=\frac{\partial g(z,t)}{\partial t}$},  $\mathcal{P}(x(t),t)$ and $\mathcal{Q}(x(t),t)$ are matrices related to the Jacobian of $g(x,t)$ that will be derived in detail later. A system of the form \eqref{eq: ODE3} is called a \textbf{time-varying projected gradient system with inertia $\alpha$}. The behavior of the solutions of this system initialized at different points depends on the value of $\alpha$. In the unconstrained case, this ODE reduces to the  \textbf{time-varying gradient system with inertia $\alpha$} given as 
\begin{equation}  \label{eq: ODE2}
 \dot{x}\comment{(t)}= -\frac{1}{\alpha}\nabla_x f(x,t)
      \tag{ODE}
\end{equation}
In what follows, we offer a motivating example without constraints (to simplify the visualization) before stating the goals of this paper.

\subsection{Motivating example}
\begin{example} \label{ex:bifurcation}
\emph{Consider $f(x,t):=\bar{f}(x-b \sin(t))$, where
\[ \bar{f}(y):=\frac{1}{4}y^4+\frac{2}{3}y^3-\frac{1}{2}y^2-2y\]
This time-varying objective has a spurious (non-global) local minimum trajectory at $-2+b\sin(t)$, a local maximum trajectory at $-1+b\sin(t)$, and a global minimum trajectory at $1+b\sin(t)$. In Figure \ref{fig:bifurcation}, we show a bifurcation phenomenon numerically. The red lines are the solutions of \eqref{eq: ODE3} with the initial point $-2$. In the case with $\alpha=0.3$ and $b=5$, the solution of \eqref{eq: ODE3} winds up in the region of attraction of the global minimum trajectory. However, for the case with $\alpha=0.1$ and $b=5$, the solution of \eqref{eq: ODE3} remains in the region of attraction of the spurious local minimum trajectory.
In the case with $\alpha=0.8$ and $b=5$, the solution of \eqref{eq: ODE3} fails to track any local minimum trajectory.
In the case with $\alpha=0.1$ and $b=10$, the solution of \eqref{eq: ODE3} winds up in the region of attraction of the global minimum trajectory.}
\begin{figure*}[ht]
\centering
\subfloat[$\alpha=0.3,b=5$]{\label{fig:bifurcation04}\includegraphics[width=0.24 \linewidth]{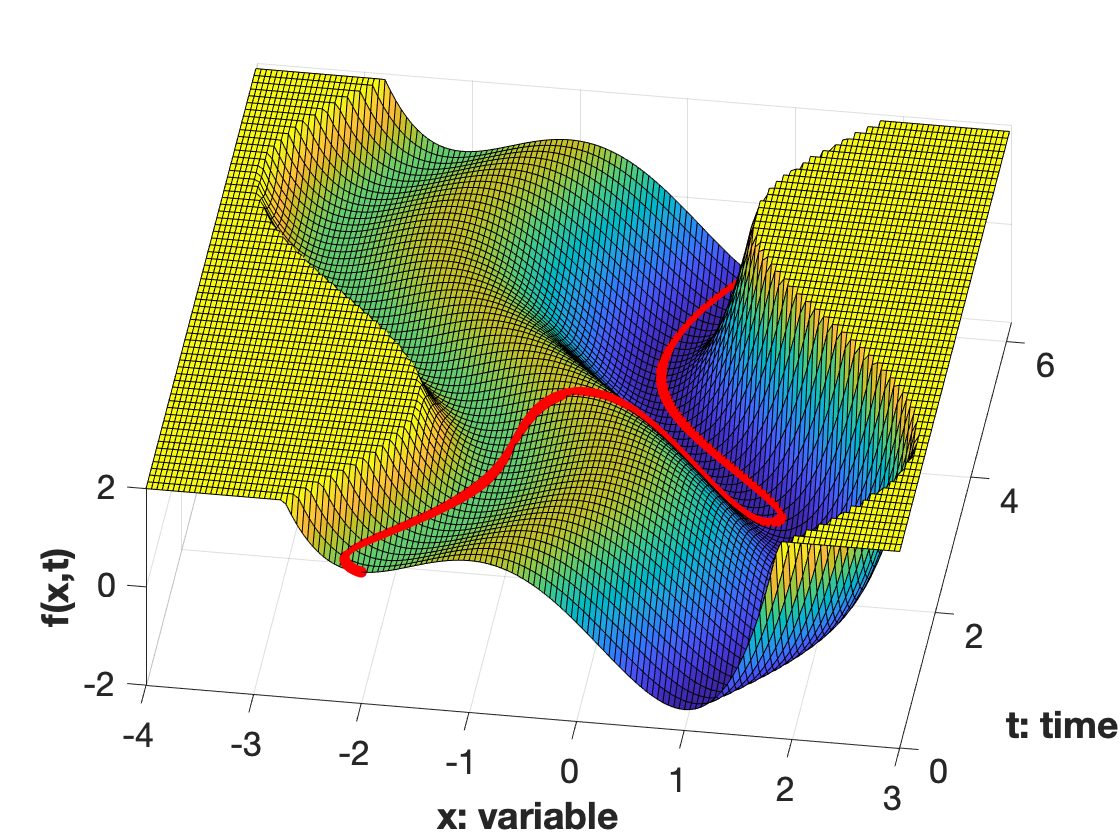}}
\subfloat[$\alpha=0.1,b=5$]{\label{fig:bifurcation02}\includegraphics[width=0.24 \linewidth]{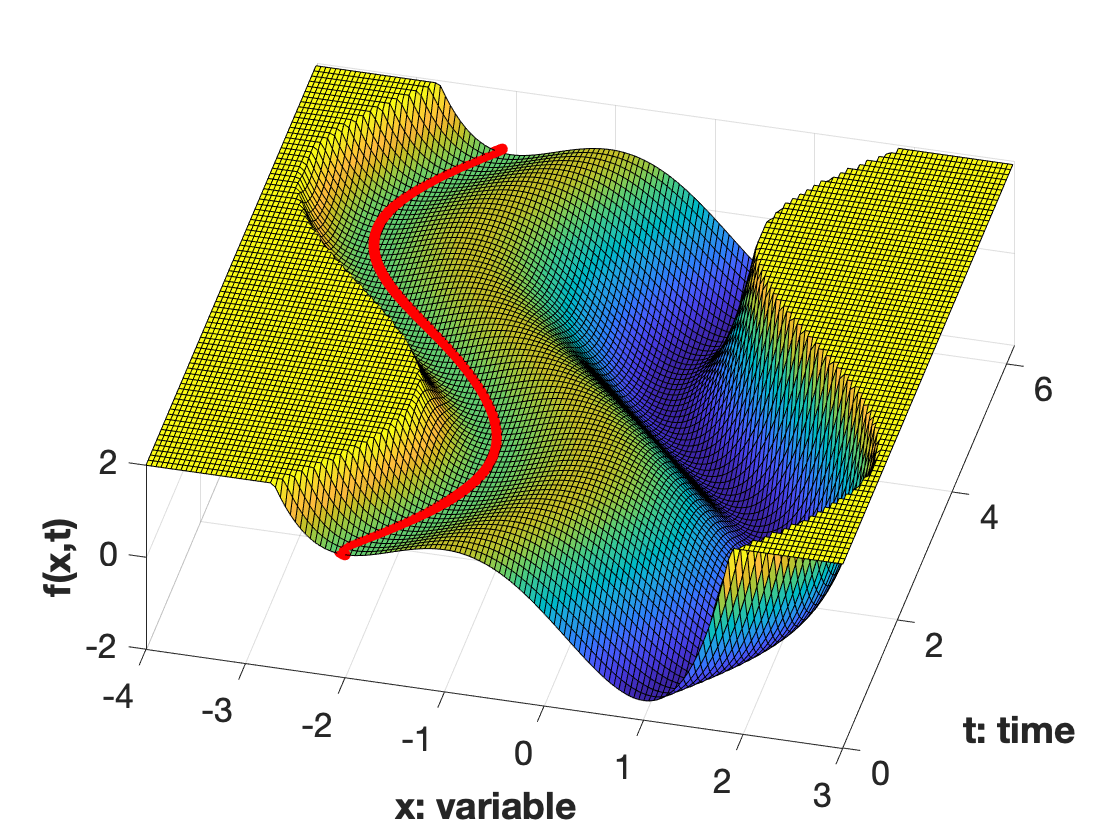}}
\subfloat[$\alpha=0.8,b=5$]{\label{fig:bifurcation_x1}\includegraphics[width=0.24 \linewidth]{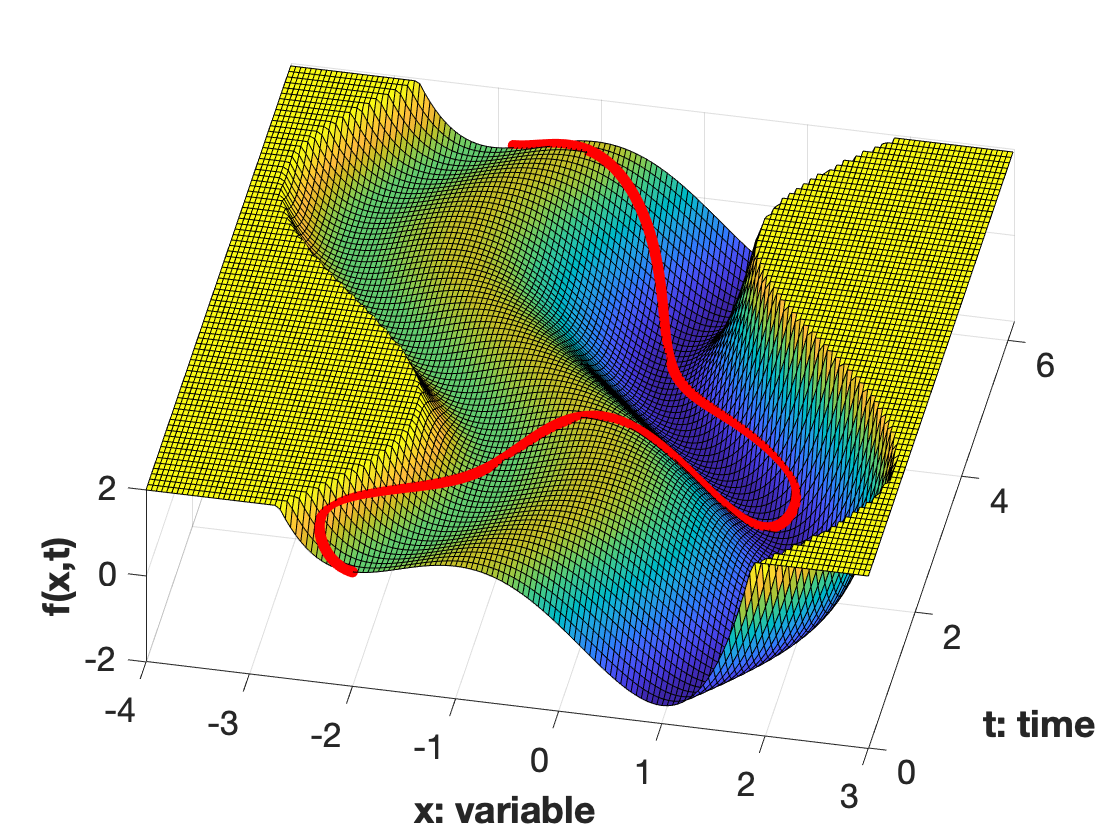}}
\subfloat[$\alpha=0.1,b=10$]{\label{fig:bifurcation_b20}\includegraphics[width=0.24 \linewidth]{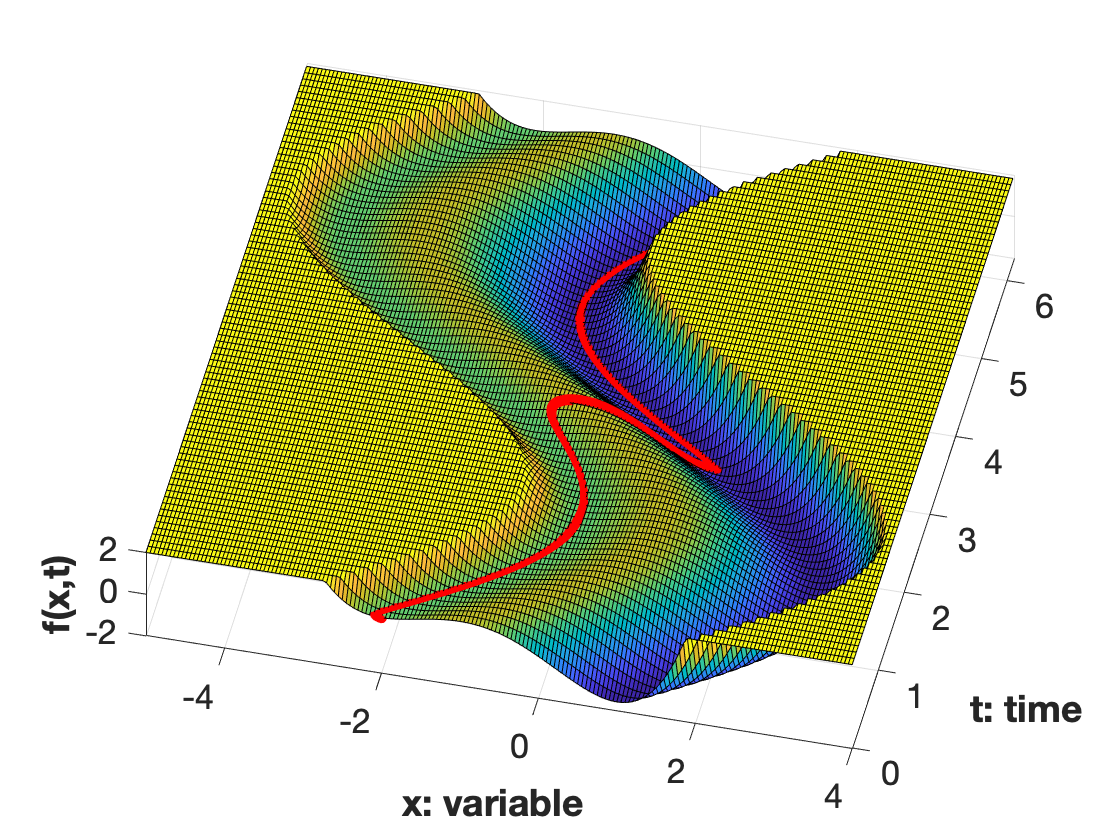}}
\caption{Illustration of Example \ref{ex:bifurcation} (in order to increase visibility, the objective function values are rescaled). Jumping from a spurious local minimum trajectory to a global minimum trajectory occurs in Figure \ref{fig:bifurcation04} and \ref{fig:bifurcation_b20} when the inertia $\alpha$ and the change (controlled by the parameter $b$) of local minimum trajectory are appropriate.}
\label{fig:bifurcation}
\end{figure*}

\emph{
Two observations can be made here. First, jumping from a local minimum trajectory to a better trajectory tends to occur with the help of a relatively large inertia when the local minimum trajectory changes the direction abruptly and  there happens to exist a better local minimum trajectory in the direction of the inertia. Second, when the inertia $\alpha$ is relatively small, the solution of \eqref{eq: ODE3} tends to track a local (or global) minimum trajectory closely and converges to that trajectory quickly.}
\end{example}
\comment{
\begin{example} \label{ex:power}
\emph{Consider the time-varying optimal power flow (OPF) problem, as the most fundamental problem for the operation of electric power grids that aims to match supply with demand while satisfying network and physical constraints. Let $f(x,t)$ be the function to be minimized at time $t$, which is the sum of the total energy cost and a penalty term taking care of all the inequality constraints of the problem. Let $g(x,t)=0$ describe the time-varying demand constraint. Assume that the load data corresponds to the California data for August 2019. As discussed in \cite{Julie2020}, this time-varying OPF has 16 local minima at t=0 and many more for some values of $t>0$. However, if (ODE) is run from any of these local minima, the 16 trajectories will all converge to the globally optimal trajectory, as shown in Figure \ref{fig:power}. This observation has been made in \cite{Julie2020} for a discrete-time version of the problem, but it also holds true for the continuous-time (ODE) model. }
\begin{figure}[ht]
\centering
  \includegraphics[width=0.8 \linewidth]{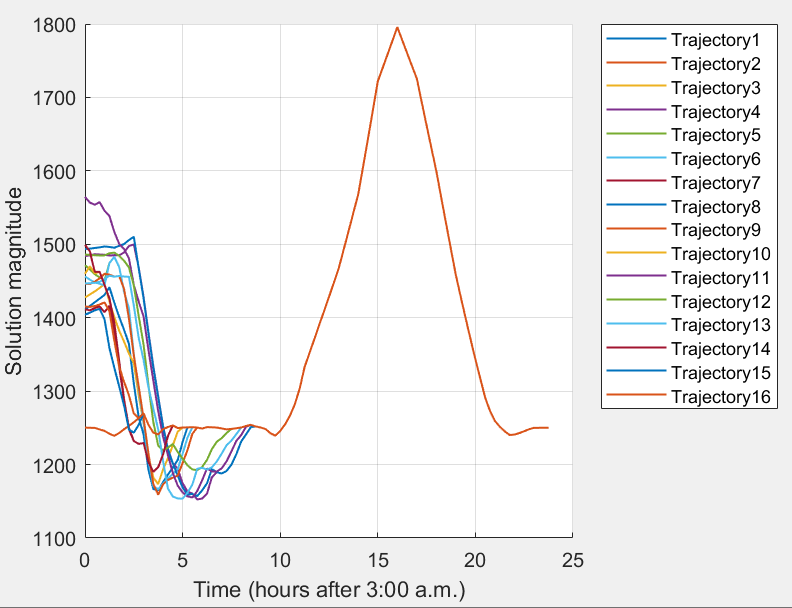}
  \caption{$|x(t)|$ (magnitude of the solution of \eqref{eq: ODE2}).}
  \label{fig:power}
\end{figure}
\end{example}}
\subsection{Our contributions}
To mathematically study the observations made in Example \ref{ex:bifurcation} and Example \ref{ex:power} for a general time-varying nonconvex optimization problem with equality constraints,  we focus on the aforementioned time-varying projected gradient flow system with inertia $\alpha$ as a continuous-time limit of an online updating scheme for \eqref{eq: constrained time-varying optimization}. 
We first introduce a time-varying Lagrange functional to unify the analysis of unconstrained problems and 
equality-constrained problems, and make the key assumption that the time-varying Lagrange functional is locally one-point \comment{strongly convex} around each local minimum trajectory. This assumption is justified by the second-order sufficient optimality conditions.
A key property of \eqref{eq: ODE3} is that its solution will remain in the time-varying feasible region if the initial point is feasible for \eqref{eq: constrained time-varying optimization}, which allows us to use the Lyapunov technique without worrying about the feasibility of the solution. 
Then, we show that the time-varying projected gradient flow system with inertia $\alpha$ is a continuous-time limit of the Karush–Kuhn–Tucker (KKT) optimality conditions for a discretized sequential optimization problem with a proximal regularization. The existence and uniqueness of the solution for such ODE is proven.
    
\comment{As a main result of this work, it is proven that the natural temporal variation of the time-varying optimization problem encourages the exploration of the state space and re-shaping the landscape of the objective function (in the unconstrained case) or the Langrange functional (in the constrained case) by making it one-point strongly convex over a large region during some time interval. 
We introduce the notion of the dominant trajectory and show that if a given spurious local minimum trajectory is dominated by the global minimum trajectory, then the temporal variation of the time-varying optimization would trigger escaping the spurious local minimum trajectory for free. }
We develop two sufficient conditions under which the ODE solution will jump from a certain local minimum trajectory to a more desirable local minimum trajectory. We then derive sufficient conditions on the inertia $\alpha$ to guarantee that the solution of \eqref{eq: ODE3} can track a global minimum trajectory.
To illustrate how  the time variation nature of an online optimization problem promotes escaping a spurious minimum trajectory, we offer a case study  with many shallow minimum trajectories. 
\subsection{Related work}
\textbf{Online time-varying optimization problems}: Time-varying optimization problems of the form \eqref{eq: constrained time-varying optimization} arise in the real-time optimal power flow problem \cite{tang2017real,hauswirth2018time} for which the power loads and renewable generations are time-varying and operational decisions should be made every 5 minutes, as well as in the real-time estimation of the state of a nonlinear dynamic system \cite{rao2003constrained}.
Other examples include model predictive control \cite{binder2001introduction}, time-varying compressive sensing \cite{salman2013sparse,balavoine2014discrete} and online economic optimization \cite{kadam2003towards,zavala2009line}.
There are many researches on the design of efficient online algorithms for tracking the optimizers of time-varying convex optimization problems \cite{simonetto2016class,fazlyab2016self,bernstein1804online,simonetto2017time}. With respect to time-varying nonconvex optimization problems, the work \cite{guddat1990parametric} presents a comprehensive theory on the structure and singularity of the KKT trajectories for time-varying optimization problems. On the algorithm side, \cite{tang2017real} provides regret-type results in the case where the constraints are lifted to the objective function via penalty functions. \cite{tang2018running} develops a running regularized primal-dual gradient algorithm to track a KKT trajectory, and offers asymptotic bounds on the tracking error. \cite{massicot2019line} obtains an ODE to approximate the KKT trajectory and derives an algorithm based on a predictor-corrector method to track the ODE solution. 

Recently, \cite{fattahi2019absence} proposed the question of whether the natural temporal variation in a time-varying nonconvex optimization problem could help a local tracking method escape spurious local minimum trajectories. 
It developed a differential equation to characterize this phenomenon (which is the basis of the current work), but it lacked mathematical conditions to guarantee this desirable behavior. The paper \cite{Julie2020} also studies this phenomenon in the context of power systems and verifies on real data for California that the natural load variation enables escaping local minima of the optimal power flow problem. The current work significantly generalizes the results of \cite{fattahi2019absence} and \cite{Julie2020} by mathematically studying when such an escaping is possible. 

\vspace{2mm}

\noindent\textbf{Local search methods for global optimization}:  Nonconvexity is inherent  in
many real-world problems:  the classical compressive sensing and matrix completion/sensing \cite{donoho2006most,candes2010matrix,candes2009exact},  training of deep neural networks \cite{li2018visualizing}, the optimal power flow problem \cite{low2014convex}, and others.
From the classical complexity theory, this nonconvexity is perceived to be the main contributor to the intractability of these problems. However, it has been recently shown that  simple local search methods, such as gradient-based algorithms, have a superb performance in solving nonconvex optimization problems. For example, \cite{lee2017first} shows that the gradient descent with a random initialization could avoid the saddle points almost surely, and \cite{jin2017escape} and \cite{ge2015escaping} prove that a perturbed gradient descent and SGD could escape the saddle points efficiently. Furthermore, it has been shown that nearly-isotropic classes of problems in matrix completion/sensing \cite{bhojanapalli2016global,ge2016matrix,zhang2019sharp}, robust principle component analysis \cite{fattahi2018exact,josz2018theory}, and dictionary recovery \cite{sun2016complete} have benign landscape, implying that they are free of spurious local minima. The work \cite{kleinberg2018alternative} proves that SGD could help escape sharp local minima of a loss function by taking the alternative view that SGD works on a convolved (thus smoothed) version of the loss function. However, these results are all for  time-invariant optimization problems for which the landscape is time-invariant. In contrast, many real-world problems should be solved sequentially over time with time-varying data. Therefore, it is essential to study the effect of the temporal variation on the landscape of  time-varying nonconvex optimization problems.

\vspace{2mm}
\noindent\textbf{Continuous-time interpretation of discrete numerical algorithms}: Many iterative numerical optimization algorithms for time-invariant optimization problems can be interpreted as a discretization of a continuous-time process. Then, several new insights have been obtained due to the known results for continuous-time dynamical systems \cite{khalil2002nonlinear,hale1980ODE}. Perhaps, the simplest and oldest example is the gradient flow system for the gradient descent algorithm with an infinitesimally small step size. The recent papers \cite{su2014differential,krichene2015accelerated,wibisono2016variational} study accelerated gradient methods for convex optimization problems from a continuous-time perspective. In addition, the continuous-time limit of the gradient descent is also employed to analyze various non-convex optimization problems, such as deep linear neural networks \cite{saxe2013exact} and matrix regression \cite{gunasekar2017implicit}. It is natural to analyze the continuous-time limit of an online algorithm for tracking a KKT trajectory of  time-varying optimization problem \cite{simonetto2016class,tang2018running,massicot2019line,fattahi2019absence}. 
\subsection{Paper organization}
This paper is organized as follows. Section \ref{sec: problem formulation} presents some preliminaries for time-varying optimization with equality constraints and the derivation of time-varying projected gradient flow with inertia. 
Section \ref{sec: change of variable} offers an alternative view on the landscape of time-varying nonconvex optimization problems after a change of variables and explains the role of the time variation of the constraints. 
Section \ref{sec: main results} analyzes the jumping, tracking and escaping behaviors of local minimum trajectories. 
Section \ref{sec:numercial example} illustrates the phenomenon that the time variation of an online optimization problem can assist with escaping spurious local minimum trajectories, by working on a benchmark example with many shallow minimum trajectories. Concluding remarks are given in Section \ref{sec:conclusion}. 
\subsection{Notation}
The notation $\mynorm{\cdot}$ represents the Euclidean norm.  The interior of the interval $\bar{I}_{t,2}$ is denoted by $\text{int}(\bar{I}_{t,2})$. The symbol $\mathcal{B}_r(h(t))=\{x\in \mathbb{R}^n: \mynorm{x-h(t)}\leq r\}$ denotes the region centered around a trajectory $h(t)$ with radius $r$ at time $t$. We denote the solution of $\dot{x}=f(x,t)$ starting from $x_0$ at the initial time $t_0$ with $x(t,t_0,x_0)$ or the short-hand notation $x(t)$ if the initial condition $(t_0,x_0)$ is clear from the context.

\section{Preliminaries and Problem Formulation} \label{sec: problem formulation}
\subsection{Time-varying optimization with equality constraints}
The first-order KKT conditions for the time-varying optimization \eqref{eq: constrained time-varying optimization} are as follows:
\begin{subequations}\label{eq: constrained stationary condition}
\begin{align} 
    0=&\nabla_x f(x(t),t)+\mathcal{J}_g(x(t),t)^\top \lambda(t)\label{eq: constrained stationary condition 1}\\
0=&g(x(t),t) \label{eq: constrained stationary condition 2}
\end{align}  
\end{subequations}
where \comment{$\mathcal{J}_g(z,t)\coloneqq\frac{\partial g(z,t) }{\partial z}$} denotes the Jacobian of $g(\cdot,\cdot)$ with respect to the first argument and $\lambda(t)\in \mathbb{R}^m$ is a Lagrange multiplier associated with the equality constraint. We first make some assumptions below.
\begin{assumption}\label{assumption:smoothness}
$f: \mathbb{R}^n \times [0, \infty) \rightarrow \mathbb{R}$ is twice continuously differentiable in $x\in\mathbb{R}^n$ and continuously differentiable in $t\geq 0$.
$g_k: \mathbb{R}^n \times [0, \infty) \rightarrow \mathbb{R}$ is twice continuously differentiable in $x\in\mathbb{R}^n$ and twice continuously differentiable in $t\geq 0$ for $k=1,\ldots,m$. Moreover, at any given time $t$, $f(x,t)$ is uniformly bounded from below over the set $\{x\in\mathbb{R}^n:g(x,t)=0\}$, meaning that there exists a constant $M$ such that $f(x,t)\geq M$  for all $x\in \{x\in\mathbb{R}^n:g(x,t)=0\}$ and $t\geq0$.
\end{assumption}
\begin{assumption} \label{assumption: feasible}
The feasible set at $t$ defined as
\begin{equation*}
    \mathcal{M}(t)\coloneqq\{x\in \mathbb{R}^n: g(x,t)=0\}
\end{equation*}
is nonempty for all $t\geq 0$.
\end{assumption}
\comment{
\begin{assumption} \label{assumption: regular of constraints}
For all $t\geq0$ and $x\in \mathcal{M}(t)$, the matrix $\mathcal{J}_g(x,t)$ has full row-rank.
\end{assumption}}
\comment{
\begin{remark}
Although Assumption \ref{assumption: regular of constraints} is somewhat stronger than the Linear independence constraint qualification \cite{bertsekas2016nonlinear}, it is necessary for our following analysis because with different values of $\alpha$ and different initial points, the solution of \eqref{eq: ODE3} may land anywhere in the feasible region. Furthermore, Sard’s theorem \cite{sard1942measure} ensures that if the constraint function $g(\cdot,t)$ is sufficiently smooth, then the set of values of $g(\cdot,t)$, denoted as $\mathcal{S}$(t), for which $\mathcal{J}_g(x,t)$ is not full row-rank has measure 0. Thus, Assumption 3 is satisfied if $0 \notin \mathcal{S}(t)$ where $\mathcal{S}(t)$ is only a set with measure 0.
Finally, if the inertia parameter $\alpha$ is fixed and the initial point of \eqref{eq: ODE3} is a local solution, then the work \cite{fattahi2019absence} provides a sophisticated proof for the existence and uniqueness of the solution for a special class of \eqref{eq: ODE3} under a minor assumption that the Jacobian has full-row rank only {at the discrete local trajectories} (which is defined in the paragraph after equation \eqref{eq: regularized constrained opt} in our work). However, to be able to study the solution of \eqref{eq: ODE3} for all $\alpha>0$ and any initial feasible point and keep the focus of the paper on studying the escaping behavior, we made Assumption 3.
\end{remark}}
Under Assumption \ref{assumption: regular of constraints}, the matrix $\mathcal{J}_g(x(t),t)\mathcal{J}_g(x(t),t)^\top$ is invertible and therefore $\lambda(t)$ in \eqref{eq: constrained stationary condition 1} can be written as 
\begin{align} \label{eq: lambda}
\lambda(t)&= - (\mathcal{J}_g(x(t),t) \mathcal{J}_g(x(t),t)^\top)^{-1}\mathcal{J}_g(x(t),t) \nabla_x f(x(t),t)
\end{align} 
Since $\lambda(t)$ is written as a function of $x(t)$ in \eqref{eq: lambda}, we also denote it as $\lambda(x(t),t)$. Now, \eqref{eq: constrained stationary condition 1} can be written as 
\begin{align}  \label{eq: constrained stationary condition 1 after replacing lambda}
\nonumber0= &\Big[I_n-\mathcal{J}_g(x(t),t)^\top(\mathcal{J}_g(x(t),t) \mathcal{J}_g(x(t),t)^\top)^{-1}\\&\mathcal{J}_g(x(t),t)\Big] \nabla_x f(x(t),t)
\end{align}
where $I_n$ is the identity matrix in $\mathbb{R}^{n\times n}$. For the sake of readability, we introduce the symbolic notation
\begin{align*}
\mathcal{P}(x(t),t)\coloneqq& I_n-\mathcal{J}_g(x(t),t)^\top(\mathcal{J}_g(x(t),t) \mathcal{J}_g(x(t),t)^\top)^{-1}\\
&\mathcal{J}_g(x(t),t)
\end{align*}
which is the orthogonal projection operation onto $T_x^t$, where $T_x^t$ denotes the tangent plane of  $g(x(t),t)$ at the point $x(t)$ and the time $t$. 
It is convenient and conventional to introduce the time-varying Lagrange functional
\begin{equation} \label{eq: lagrange with lambda}
    L(x,\lambda,t)= f(x,t) +\lambda g(x,t)
\end{equation}
In terms of this functional, \eqref{eq: constrained stationary condition 1 after replacing lambda} can be written as 
\begin{equation}
    0=\nabla_x L(x,\lambda,t)
\end{equation}
where $\lambda$ is given in \eqref{eq: lambda}. \comment{Here, $\nabla_x L(x,\lambda,t)$ means first taking the partial gradient with respect to the first argument and then using the formula \eqref{eq: lambda} for $\lambda$.}
\comment{Since the solution is time-varying, we define the notion of the local (or global) minimum trajectory below.
\begin{definition}
A continuous trajectory $h: I_t \rightarrow \mathbb{R}^n$, where $I_t\subseteq [0,\infty)$, is said to be a \textbf{local (or global) minimum trajectory} of the time-varying optimization \eqref{eq: constrained time-varying optimization} if each point of $h(t)$ is a local (or global) minimum of the time varying optimization \eqref{eq: constrained time-varying optimization} for every $t \in I_t$. 
\end{definition}
}
\comment{
In this paper, we focus on the case when the local minimum trajectories will not cross, bifurcate or disappear by assuming the following uniform regularity condition.
\begin{assumption} \label{ass: non-singular}
For each local minimum trajectory $h(t)$, its domain $I_t$ is $[0,\infty)$ and $h(t)$ satisfies the second-order sufficient optimality conditions uniformly, meaning that $\nabla_{xx}^2 L(h(t),\lambda,t)$ is positive definite on $T_{h(t)}^t=\{y:\mathcal{J}_g(h(t),t)^\top y=0\}$ for all $t \in [0,\infty)$.
\end{assumption}
\begin{lemma} \label{lemma: differentiable and isolated}
Under Assumptions \ref{assumption:smoothness}-\ref{ass: non-singular}, each local minimum trajectory $h(t)$ is differentiable and isolated, and therefore it can not bifurcate or merge with other local minimum trajectories.
\end{lemma}
\begin{proof}
Under Assumptions \ref{assumption:smoothness}-\ref{ass: non-singular}, one can apply the inverse function theorem to \eqref{eq: constrained stationary condition} (see \cite[Theorem~4.4, Example~4.7]{still2018lectures}) to conclude that for every $h(\bar t)$ and $\bar t$, there exist an open set $\mathcal{S}_{h(\bar t)}$ containing $h(\bar t)$ and an open set $\mathcal{S}_{\bar t}$ containing $\bar t$ such that there exist a unique differentiable function $x(t)$ in $\mathcal{S}_{h(\bar t)}$  for all $t\in \mathcal{S}_{\bar t}$ where $x(t)$ is the isolated local minimizer of the time-varying optimization problem \eqref{eq: constrained time-varying optimization}. Because of this uniqueness property and the continuity of the local minimum trajectory $h(t)$, $x(t)$ must coincide with $h(t)$ for all $t\in \mathcal{S}_{\bar t}$. Then, because the above property holds uniformly for every $t \in [0,\infty)$, $h(t)$ must be a differentiable isolated minimum trajectory.
\end{proof}
}
After freezing the time $t$ in \eqref{eq: constrained time-varying optimization} at a particular value, one may use local search methods, like Rosen's gradient projection method \cite{rosen1961gradient}, to minimize $f(x,t)$ over the feasible region $\mathcal{M}(t)$. If the initial point is feasible and close enough to a local solution and the step size is small enough, the algorithm will converge to the local minimum. This leads to the notion of region of attraction defined by resorting to the continuous-time model of Rosen's gradient projection method \cite{tanabe1974algorithm} (for which the step size is not important anymore).
\begin{definition}
The \textbf{region of attraction} of a local minimum point $h(t)$ of $f(\cdot,t)$ {in the feasible set $\mathcal{M}(t)$} at a given time $t$ is defined as:
\begin{align*}
&RA^{\mathcal{M}(t)}(h(t))=\big\{x_0\in\mathcal{M}(t)\ \big|\ \lim_{\tilde t \rightarrow \infty}\tilde x(\tilde t)=h(t) \quad \text{where} \quad\\
&\frac{d \tilde x(\tilde t)}{d \tilde t}= -\mathcal{P}(\tilde x(\tilde t),t)\nabla_x f(\tilde x(\tilde t),t) \quad \text{and} \quad \tilde x(0)=x_0 \big\}
\end{align*}
\end{definition}
In the unconstrained case, the notion of the {locally one-point strong convexity} can be defined as follows:

\begin{definition}  \label{def: strong convex unconstrained}
Consider arbitrary positive scalars $c$ and $r$. The function $f(x,t)$ is said to be \comment{\textbf{locally $(c,r)$-one-point strongly convex} around the local minimum trajectory $h(t)$ if }
\begin{equation} \label{eq: one-point_1}
       \nabla_x f(e+h(t),t)^\top e \geq c \mynorm{e}^2, \quad \forall e \in D,\quad  \comment{\forall t\in [0,\infty)}
\end{equation}
where $D=\{e\in \mathbb{R}^n: \mynorm{e}\leq r\}$. The region $D=\{e\in \mathbb{R}^n: \mynorm{e}\leq r\}$ is called the  \comment{{region of locally $(c,r)$-one-point strong convexity}} around $h(t)$.
\end{definition}

This definition resembles the (locally) strong convexity condition for the function $f(x,t)$, but it is only expressed around the point $h(t)$. This restriction to a single point constitutes the definition of one-point strong convexity and it does not imply that the function is convex. 
The following result paves the way for the generalization of the notion of the {locally one-point strong convexity} from the unconstrained case to the equality constrained case.

\begin{lemma} \label{lemma: there exist r st one-point convex}
\comment{Consider an arbitrary local minimum trajectory $h(t)$ satisfying Assumption \ref{ass: non-singular},} there exist positive constants $\hat{r}$ and $\hat{c}$ such that 
\begin{equation*}
    e(t)^\top  \nabla_x L(e(t)+h(t),\lambda(e(t)+h(t),t),t)\geq \hat{c} \mynorm{e(t)}^2
\end{equation*}
for all $e(t) \in \{e+h(t)\in \mathcal{M}(t): \mynorm{e}\leq \hat{r}\}$.
\end{lemma}
\begin{proof}
Due to the second-order sufficient conditions for the equality constrained minimization problem, $\nabla_{xx}^2 L(h(t),\lambda(h(t),t),t)$ is positive definite on $T_{h(t)}^t$ for all \comment{$t\in [0,\infty)$}, meaning that for every nonzero vector $y\in T_{h(t)}^t$, there exists a positive constant $\bar{c}$ such that $y\nabla_{xx}^2 L(h(t),\lambda,t)y>\bar{c}\mynorm{y}^2$. Since $\mathcal{P}(h(t),t)$ is the orthogonal projection matrix onto the tangent plane $T_{h(t)}^t$, we have $y\nabla_{xx}^2 L(h(t),\lambda(h(t),t),t) \mathcal{P}(h(t),t) y>\bar{c}\mynorm{y}^2$ for all $y\in T_{h(t)}^t$ and $y\neq 0$, and 
$y\nabla_{xx}^2 L(h(t),\lambda(h(t),t),t)$ $\mathcal{P}(h(t),t) y=0$ for all $y\notin T_{h(t)}^t$.
Taking the first-order Taylor expansion of $\nabla_x L(x,\lambda(x,t),t)$ with respect to $x$ around $h(t)$ and using the following result from \cite[Corollary 1]{luenberger1972gradient}:
\begin{align*}
\evalat{\frac{\partial }{\partial x}\nabla_x L(x,\lambda(x,t),t)}{x=h(t)}=& {\nabla_{xx}^2L(h(t),\lambda(h(t),t),t)}\\
&\mathcal{P}(h(t),t),
\end{align*}
it yields that 
\begin{align*}
&e(t)^\top  \nabla_x L(e(t)+h(t),\lambda,t)=e(t)^\top  \nabla_x L(h(t),\lambda,t)\\
+&e(t)^\top  \nabla_{xx}^2 L(h(t),\lambda,t)\mathcal{P}(h(t),t) e(t)+o(e(t)^3)\\
=&e(t)^\top  \nabla_{xx}^2 L(h(t),\lambda,t)\mathcal{P}(h(t),t)e(t)+o(e(t)^3)
\end{align*}
From Lemma \ref{lemma: local lip of ODE3} in Appendix \ref{appendix: local lip of ODE3}, we know that $\nabla_{xx}^2 L(x,\lambda,t)\mathcal{P}(x,t)$ is continuous in $x$ and $t$. In addition, $g(x,t)$ is also continuous in $x$ and $t$. As a result, there exist positive constants $\hat{r}$ and $\hat{c}$ such that
\begin{equation*}
    e(t)^\top  \nabla_x L(e(t)+h(t),\lambda,t)\geq \hat{c} \mynorm{e(t)}^2
\end{equation*}
for all $e(t) \in \{e+h(t)\in \mathcal{M}(t): \mynorm{e}\leq \hat{r}\}$
\end{proof}

\begin{definition}  \label{def: strong convex on manifold}
Consider arbitrary positive scalars $c$ and $r$.
The Lagrange function $L(x,\lambda,t)$ with $\lambda$ given in \eqref{eq: lambda} is said to be \comment{\textbf{locally $(c,r)$-one-point strongly convex} with respect to $x$ around the local minimum trajectory $h(t)$ in the feasible set $\mathcal{M}(t)$} if:
\begin{equation} \label{eq: one-point_1 on manifold}
  e^\top  \nabla_x L(e+h(t),\lambda(e+h(t),t),t) \geq c \mynorm{e}^2 
\end{equation}
for all $e \in D^{\mathcal{M}(t)}$ and \comment{$t\in [0,\infty)$}, where $D^{\mathcal{M}(t)}=\{e+h(t)\in \mathcal{M}(t): \mynorm{e}\leq r\}$. The region $D^{\mathcal{M}(t)}=\{e+h(t)\in \mathcal{M}(t): \mynorm{e}\leq r\}$ is called the \comment{region of locally $(c,r)$-one-point strong convexity} of the Lagrange function $L(x,\lambda,t)$ around $h(t)$ in the feasible set $\mathcal{M}(t)$.
\end{definition}
\begin{remark}
The Lagrange function $L(x,\lambda,t)$ with $\lambda$ given in \eqref{eq: lambda} being \comment{locally $(c,r)$-one-point strongly convex with respect to $x$} around $h(t)$ is equivalent to the vector field $\mathcal{P}(x,t)\nabla_x f(x(t),t)$ being \comment{\textbf{locally $(c,r)$-one-point strongly monotone}} with respect to $x$ around $h(t)$.
\end{remark}



\subsection{Derivation of time-varying projected gradient flow system}
In practice, one can only hope to sequentially solve the time-varying optimization problem \eqref{eq: constrained time-varying optimization} at some discrete time instances $0 = \tau_0 < \tau_1 < \tau_2 <\tau_3< \ldots$ as follows:
\begin{equation} \label{eq: sequential constrained opt}
  \min_{x\in \mathbb{R}^n} \quad  f(x,\tau_i), \quad   \text{s.t.} \quad  g(x,\tau_i)=0,\quad i=1,2,\ldots 
\end{equation} 

In many real-world applications, it is neither practical nor realistic to have solutions that abruptly change over time. To meet this requirement, we impose a soft constraint to the objective function by penalizing the deviation of its solution from the one obtained in the previous time step. 
This leads to the following sequence of optimization problems with \textbf{proximal regularization} (except for the initial optimization problem):
\begin{subequations}
\label{eq: regularized constrained opt}
\begin{align} 
  \min_{x\in \mathbb{R}^n} \quad & f(x,\tau_0), \label{eq: regularized constrained opt t0} \\
  \nonumber \text{s.t.} \quad  &g(x,\tau_0)=0,  \\
 \min_{x\in \mathbb{R}^n} \quad & f(x,\tau_i) +\frac{\alpha}{2(\tau_{i}-\tau_{i-1})}\mynorm{x-x^\ast_{i-1}}^2, \label{eq: regularized constrained opt tk} \\
  \nonumber \text{s.t.} \quad & g(x,\tau_i)=0,\quad i=1,2,\ldots 
\end{align}
\end{subequations}
where $x^\ast_{i-1}$ denotes  an  arbitrary  local  minimum  of  the  modified  optimization  problem  \eqref{eq: regularized constrained opt} obtained  using  a local  search  method  at  time  iteration $i-1$. A local optimal solution sequence $x_0^\ast, x_1^\ast, x_2^\ast, \ldots$ is said to be a \textbf{discrete local trajectory} of the sequential {regularized} optimization \eqref{eq: regularized constrained opt}. \comment{The parameter $\alpha$ is called inertia because it acts as a resistance to changes $x$ at time step $\tau_i$ with respect to $x$ at the previous time step $\tau_{i-1}$.}
Note that $\alpha$ could be time-varying (and adaptively changing)  in the analysis of this paper, but we restrict our attention to a fixed regularization term to simplify the presentation. 

Under Assumption \ref{assumption: regular of constraints}, all solutions $x^\ast$ of \eqref{eq: regularized constrained opt tk} must satisfy the KKT conditions:
\begin{subequations} \label{eq: constrained KKT}
    \begin{align} 
    0&=\nabla_x f(x_i^\ast,\tau_i)+\alpha \frac{x_i^\ast-x_{i-1}^\ast}{\tau_{i}-\tau_{i-1}}+ \mathcal{J}_g(x_i,\tau_i)^\top \bar{\lambda}_i, \label{eq: constrained KKT 1}\\
    0&=g(x_i,\tau_i), \label{eq: constrained KKT 2}
\end{align}
\end{subequations}
where $\bar{\lambda}_i$'s are the Lagrange multipliers for the sequence of optimization problems with proximal regularization in \eqref{eq: regularized constrained opt}.  Similar to \cite{massicot2019line}, we can write the right-hand side of the constraint \eqref{eq: constrained KKT 2} as:
\begin{equation} 
    \frac{g(x_i,\tau_i)-g(x_i,\tau_{i-1})+g(x_i,\tau_{i-1})-g(x_{i-1},\tau_{i-1})}{\tau_i-\tau_{i-1}}
\end{equation}
Since the function $f(x,t)$ and $g(x,t)$ are nonconvex in general, the problem \eqref{eq: regularized constrained opt}  may not have a unique solution  $x_{i}^\ast$.  In order to cope with this issue, we study the continuous-time limit of \eqref{eq: constrained KKT} as the time step $\tau_{i+1}-\tau_i$ diminishes to zero. This yields the following time-varying ordinary differential equations:
\begin{subequations} \label{eq: cts time kkt}
    \begin{align} 
    0&=\nabla_x f(x(t),t)+\alpha \dot{x}(t)+ \mathcal{J}_g(x(t),t)^\top\bar{\lambda}(t), \label{eq: continuous constrained KKT 1}\\
    0&=\mathcal{J}_g(x(t),t)\dot{x}(t)+g^\prime(x(t),t), \label{eq: continuous constrained KKT 2}
\end{align}
\end{subequations}
where $g^\prime=\frac{\partial g(x,t)}{\partial t}$ denotes the partial derivative of $g$ with respect to $t$. Since $\mathcal{J}_g(x(t),t)\mathcal{J}_g(x(t),t)^\top$ is invertible, we have
\begin{align}
\nonumber0=&(\mathcal{J}_g(x(t),t)\mathcal{J}_g(x(t),t)^\top)^{-1}\mathcal{J}_g(x(t),t) \nabla_x f(x(t),t)\\
&-\alpha(\mathcal{J}_g(x(t),t)\mathcal{J}_g(x(t),t)^\top)^{-1}g^\prime(x(t),t) + \bar{\lambda}(t).
\end{align}
Therefore, $\bar{\lambda}(t)$ can be written as a function of $x$, $t$ and $\alpha$:
\begin{align}
   \nonumber \bar{\lambda}(t)=&-(\mathcal{J}_g(x(t),t)\mathcal{J}_g(x(t),t)^\top)^{-1}\mathcal{J}_g(x(t),t)\nabla_x f(x(t),t)\\
   \nonumber&+\alpha(\mathcal{J}_g(x(t),t)\mathcal{J}_g(x(t),t)^\top)^{-1}g^\prime(x(t),t)\\
   =& \lambda(x(t),t)+\alpha (\mathcal{J}_g(x,t)\mathcal{J}_g(x,t)^\top)^{-1}g^\prime (x,t)\label{eq: lambda bar}
\end{align}
We alternatively denote $\bar\lambda(t)$ as $\bar\lambda(x(t),t,\alpha)$. When $\alpha=0$, we have $\bar{\lambda}(x(t),t,\alpha)={\lambda}(x(t),t)$ and the differential equation \eqref{eq: cts time kkt} reduces to the algebraic equation \eqref{eq: constrained stationary condition}, which is indeed the first-order KKT condition for the unregularized time-varying optimization \eqref{eq: constrained time-varying optimization}. 
When $\alpha>0$, substituting $\bar{\lambda}(x(t),t,\alpha)$ into \eqref{eq: continuous constrained KKT 1} yields the following time-varying ODE:
\begin{equation}  \tag{P-ODE}
\dot{x}(t) = -\frac{1}{\alpha}\mathcal{P}(x(t),t) \nabla_x f(x(t),t) - \mathcal{Q}(x(t),t)g^\prime(x(t),t),
\end{equation}
where $\mathcal{Q}(x(t),t)=\mathcal{J}_g(x(t),t)^\top (\mathcal{J}_g(x(t),t)\mathcal{J}_g(x(t),t)^\top)^{-1}$.
In terms of the Lagrange functional, \eqref{eq: ODE3} can be written as 
\begin{equation} \label{eq: langange with lambda bar}
    \dot{x}=-\frac{1}{\alpha}\nabla_x L(x,\bar{\lambda},t)=-\frac{1}{\alpha}\nabla_x L(x,{\lambda},t)-\mathcal{Q}(x,t)g^\prime(x,t).
\end{equation}
\comment{Here, $\nabla_x L(x,\bar \lambda,t)$ means first taking the partial gradient with respect to the first argument and then using the formula \eqref{eq: lambda bar} for $\bar\lambda$.} It can be shown that if the initial point of \eqref{eq: ODE3} is in the feasible set $M(t_0)$, the solution of \eqref{eq: ODE3} will stay in the feasible set $M(t)$.
\begin{lemma} \label{lemma: invariance}
Suppose that the solution $x(t,t_0,x_0)$ of \eqref{eq: ODE3} is defined in $[t_0,\infty)$ with the initial point $x_0$. If $x_0 \in \mathcal{M}(t_0)$, then the solution $x(t,t_0,x_0)$ belongs to $\mathcal{M}(t)$ for all $t\geq t_0$.
\end{lemma}
\begin{proof}
On examining the evolution of $g(x(t),t)$ along the flow of the system \eqref{eq: ODE3}, we obtain
\begin{equation*}
   \comment{ \dot{g}(x(t),t)=\mathcal{J}_g(x(t),t)\dot{x}(t)+g^\prime(x(t),t)=0 }
\end{equation*}
Hence, $g(x(t_0),t_0)=g(x(t,t_0,x_0),t)$ for all $t\geq t_0$.
\end{proof}
Therefore, as long as the initial point of \eqref{eq: ODE3} is in the feasible set $M(t_0)$, the above lemma guarantees that we can \comment{analyze} the stability of \eqref{eq: ODE3} using the standard Lyapunov's theorem without worrying about the feasibility of the solution. 
\comment{When $\alpha>0$, we will show that for any initial point $x_0$, \eqref{eq: ODE3} has a unique solution defined for all $t\in I_t \subseteq [0,\infty)$ if there exists a local minimum trajectory $h(t)$ such that the solutions of \eqref{eq: ODE3} lie in a compact set around $h(t)$} \footnote{\comment{In Theorems \ref{thm: jump} and \ref{thm: jump average}, the compactness assumption is included in the definition of the dominant trajectory. In Theorem \ref{thm: Sufficient condition for tracking constrained}, checking the compactness assumption can be carried out via the Lyapunov's method without solving the differential equation due to the one-point strong convexity condition around $h(t)$.}}.
\comment{\begin{theorem}[Existence and uniqueness] \label{lemma: exist and unique}
Under Assumptions \ref{assumption:smoothness}-\ref{ass: non-singular} and given any initial point $x_0 \in \mathcal{M}(t_0)$, suppose that there exists a local minimum trajectory  $h(t)$ with the property that $x(t)-h(t)$ lies entirely in ${D}$ for all $t\in I_t \subseteq [0,\infty)$ where ${D}$ is a compact subset of  $\mathbb{R}^n$ containing $x_0-h(t_0)$. Then, \eqref{eq: ODE3} has a unique solution starting from $x_0$ that is defined for all $t\geq 0$.
\end{theorem}
\begin{proof}
Since $h(t)$ is differentiable by Lemma \ref{lemma: differentiable and isolated},  we can use the change of variables $e(t)=x(t)-h(t)$ to rewrite \eqref{eq: ODE3} as:
\begin{align}   \label{eq: e exist and uniqueness}
\nonumber  \dot{e}(t)=&-\frac{1}{\alpha}\mathcal{P}(e(t)+h(t),t)\nabla_x f(e(t)+h(t),t)-\\
&\mathcal{Q}(e(t)+h(t),t)g^\prime(e(t)+h(t),t) - \dot{h}(t)  
\end{align}
In light of the conditions in Theorem \ref{lemma: exist and unique}, the solution of \eqref{eq: e exist and uniqueness} stays in a compact set. Then, by Lemma \ref{lemma: invariance} and \cite[Theorem 3.3]{khalil2002nonlinear}, the equation \eqref{eq: e exist and uniqueness} has a unique solution. Thus, \eqref{eq: ODE3} must also have a unique solution.
\end{proof} }
In online optimization, it is sometimes desirable to predict the solution at a future time (namely, $\tau_i$) only based on the information at the current time (namely, $\tau_{i-1}$). This can be achieved by implementing the forward Euler method to obtain a numerical approximation to the solution of \eqref{eq: ODE3}:
\begin{align} \label{eq: forward euler}
  \nonumber\bar{x}^\ast_{i} =& \bar{x}^\ast_{i-1}-({\tau_i-\tau_{i-1}})\Big(\frac{1}{\alpha} \mathcal{P}(\bar{x}^\ast_{i-1},\tau_{i-1})\nabla_x f(\bar{x}^\ast_{i-1},\tau_{i-1})\\
  &+\mathcal{Q}(\bar{x}^\ast_{i-1},\tau_{i-1})g^\prime(\bar{x}^\ast_{i-1},\tau_{i-1})\Big)
\end{align}
(note that $\bar{x}^\ast_0,\bar{x}^\ast_1,\bar{x}^\ast_2,...$ show the approximate solutions). 
The following theorem explains the reason behind studying the continuous-time problem \eqref{eq: ODE3} in the remainder of this paper.
\begin{theorem}[Convergence] \label{lemma: ODE convergence}
Under Assumptions \ref{assumption:smoothness}-\ref{ass: non-singular} and given  a local minimum $x_0^\ast$ of \eqref{eq: regularized constrained opt t0}, as the time difference $\Delta_\tau=\tau_{i+1}-\tau_i$ approaches zero, any sequence of discrete local trajectories ($x_k^{\Delta}$) converges to the \eqref{eq: ODE3} in the sense that for all fixed $T >0$:
\begin{equation}
   \lim_{ \Delta_\tau\rightarrow 0} \max_{0\leq k \leq \frac{T}{\Delta \tau}} \mynorm{x_k^{\Delta}-x(\tau_k,\tau_0,x_0^\ast)}=0
\end{equation}
and any sequence of ($\bar{x}_k^{\Delta}$) updated by \eqref{eq: forward euler} converges to the \eqref{eq: ODE3} in the sense that for all fixed $T >0$:
\begin{equation}
   \lim_{ \Delta_\tau\rightarrow 0} \max_{0\leq k \leq \frac{T}{\Delta \tau}} \mynorm{\bar{x}_k^{\Delta}-x(\tau_k,\tau_0,x_0^\ast)}=0
\end{equation}
\end{theorem}
\begin{proof}
The first part follows from Theorem 2 in \cite{fattahi2019absence}. For the second part, a direct application of the classical results on convergence of the forward Euler method \cite{iserles2009first} immediately shows that the solution of \eqref{eq: ODE3} starting at a local minimum of \eqref{eq: regularized constrained opt t0} is the continuous limit of the discrete local trajectory of  the sequential {regularized} optimization \eqref{eq: regularized constrained opt}.
\end{proof}

Theorem \ref{lemma: ODE convergence} guarantees that the solution of \eqref{eq: ODE3} is a reasonable approximation in the sense that it is the continuous-time limit of both the solution of the sequential regularized optimization problem \eqref{eq: regularized constrained opt} and the solution of the online updating scheme \eqref{eq: forward euler}. For this reason, we only study the continuous-time problem~\eqref{eq: ODE3} in the remainder of this paper.

\subsection{Jumping, tracking and escaping}
In this paper, the objective is to study the case where there are at least \comment{two local minimum trajectories} of the online time-varying optimization problem. 
\comment{Consider two local minimum trajectories $h_1(t)$ and $h_2(t)$. }
We provide the definitions of jumping, tracking and escaping below.
\begin{definition}
It is said that the solution of \eqref{eq: ODE3} \textbf{(v,u)-jumps} from $h_1(t)$ to $h_2(t)$ over the time interval $[t_1,t_2]$ if there exist $u>0$ and $v>0$ such that
\begin{subequations}
\begin{align}
& \mathcal{B}_v(h_1(t_1))\cap\mathcal{M}(t_1)  \subseteq RA^{\mathcal{M}(t_1)}(h_1(t_1)) \\
 &\mathcal{B}_u(h_2(t_2))\cap\mathcal{M}(t_2) \subseteq RA^{\mathcal{M}(t_2)}(h_2(t_2))\\
\nonumber& \forall x_1\in \mathcal{B}_v(h_1(t_1))\cap\mathcal{M}(t_1)\\
&\Longrightarrow x(t_2,t_1,x_1)\in\mathcal{B}_u(h_2(t_2))\cap\mathcal{M}(t_2)  
\end{align}
\end{subequations}
\end{definition}

\begin{definition}
Given $x_0\in \mathcal{M}(t_0)$, it is said that $x(t,t_0,x_0)$ \textbf{u-tracks} $h_2(t)$ if there exist a finite time $T>0$ and a constant  $u>0$ such that  
\begin{subequations}
\begin{align}
&x(t,t_0,x_0)\in \mathcal{B}_u(h_2(t))\cap\mathcal{M}(t), & \forall t\geq T\\
& \mathcal{B}_u(h_2(t))\cap\mathcal{M}(t) \subseteq RA^{\mathcal{M}(t)}(h_2(t)), &\forall t\geq T
\end{align}
\end{subequations}
\end{definition}
In this paper, the objective is to study the scenario where a solution $x(t,t_0,x_0)$ tracking a poor solution $h_1(t)$ at the beginning ends up tracking a better solution $h_2(t)$ after some time. This needs the notion of ``escaping" which is a combination of jumping and tracking. 
\begin{definition}
It is said that the solution of \eqref{eq: ODE2} \textbf{(v,u)-escapes} from $h_1(t)$ to $h_2(t)$ if there exist $T>0$, $u>0$ and $v>0$ such that
\begin{subequations}
\begin{align}
& \mathcal{B}_v(h_1(t_0))\cap\mathcal{M}(t_0) \subseteq RA^{\mathcal{M}(t_0) }(h_1(t_0)) \\
&\mathcal{B}_u(h_2(t))\cap\mathcal{M}(t)\subseteq RA^{\mathcal{M}(t) }(h_2(t)), \ \comment{\forall t\geq T}\\
\nonumber & \forall x_0\in \mathcal{B}_v(h_1(t_0))\cap\mathcal{M}(t_0)\Longrightarrow\\&  x(t,t_0,x_0) \in \mathcal{B}_u(h_2(t))\cap\mathcal{M}(t), \ \comment{ \forall t \geq T }
\end{align}
\end{subequations}
\end{definition}
\begin{figure}[ht]
\centering
  \includegraphics[width=1 \linewidth]{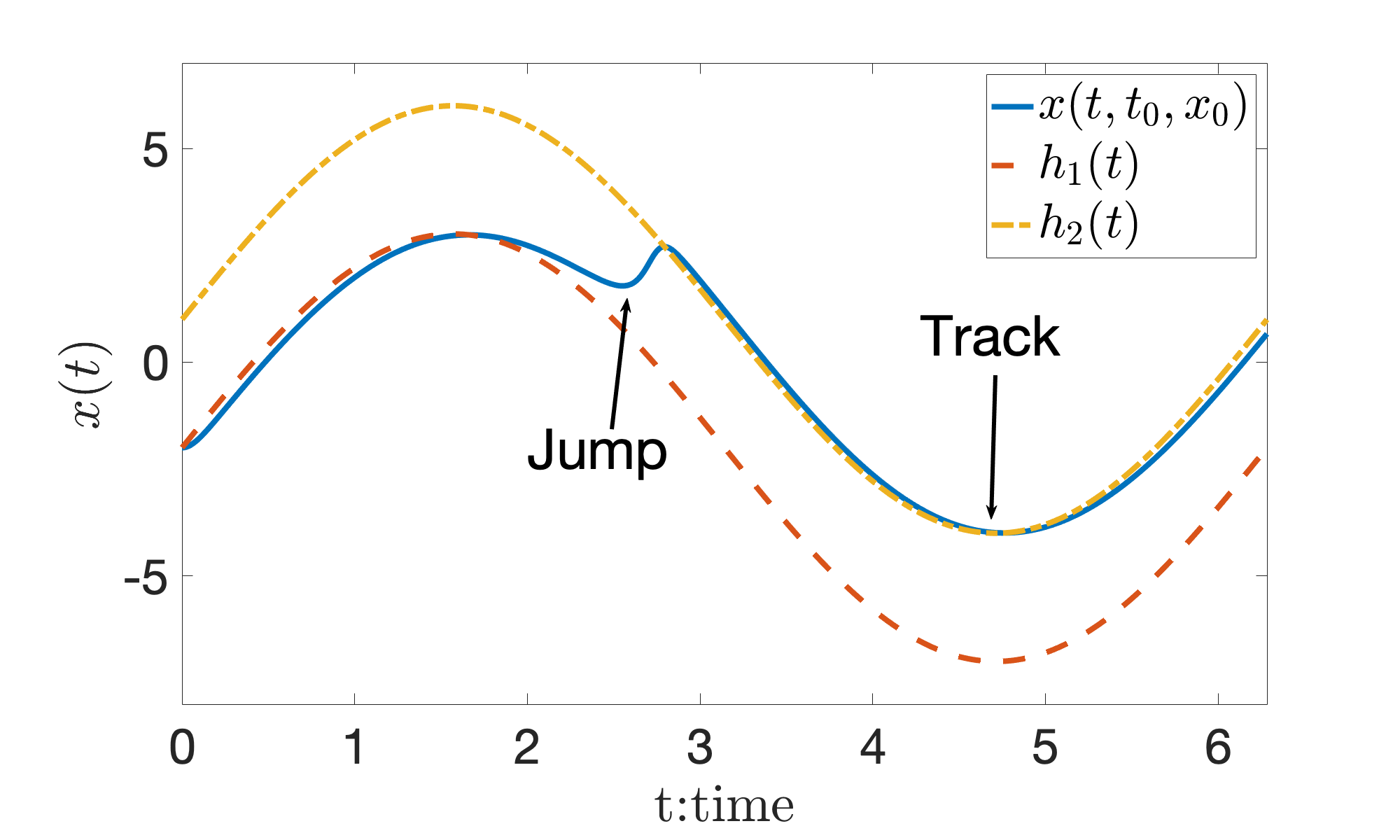}
  \caption{Illustration of jumping and tracking.}
  \label{fig: jump_track}
\end{figure}

Figure \ref{fig: jump_track} illustrates the definitions of jumping and tracking for Example \ref{ex:bifurcation} with $\alpha=0.3$ and $b=5$.
The objective of this paper is to study when the solution of \eqref{eq: ODE3} started at a poor local minimum at the initial time jumps to and tracks a better (or global) minimum of the problem after some time. In other words, it is desirable to investigate the escaping property from $h_1(t)$ and $h_2(t)$.

\section{Change of variables}\label{sec: change of variable}
Given two isolated local minimum trajectories $h_1(t)$, $h_2(t)$. One may use the change of variables  $x(t,t_0,x_0)$ $=e(t,t_0,e_0)+h_2(t)$ to transform \eqref{eq: ODE3} into the form 
\begin{subequations}\label{eq: ode3 change variable}
\begin{align}  
\nonumber  \dot{e}(t)=&-\frac{1}{\alpha}\mathcal{P}(e(t)+h_2(t),t)\nabla_x f(e(t)+h_2(t),t)-\\
   &\mathcal{Q}(e(t)+h_2(t),t)g^\prime(e(t)+h_2(t),t) - \dot{h}_2(t)  \label{eq: ode3 change variable gradient}\\
\nonumber   =&-\frac{1}{\alpha}\nabla_x \Big( L\big(e(t)+h_2(t),\bar{\lambda}(e(t)+h_2(t),t, \alpha),t\big)\\
  &+ \alpha\dot{h}_2(t)\comment{^\top} e(t) \Big) \label{eq: ode3 change variable landscape}
\end{align}
\end{subequations}
We use $e(t,t_0,e_0)$ to denote the solution of this differential equation starting at time $t=t_0$ with the initial point $e_0=x_0-h_2(t_0)$ \comment{and use $-\frac{1}{\alpha}U(e(t),t,\alpha)$ to denote the right-hand side of \eqref{eq: ode3 change variable}}.  Note that $h_1(t)$ and $h_2(t)$ are local solutions of \eqref{eq: constrained time-varying optimization} and as long as \eqref{eq: constrained time-varying optimization} is time-varying, these functions cannot satisfy \eqref{eq: ODE3} in general.  We denote $\mathcal{M}^h(t)\coloneqq\{e\in \mathbb{R}^n: g(e+h(t),t)=0\}$.
\subsection{Unconstrained optimization landscape after a change of variables} \label{sec: alternative view}
In this subsection, we study the unconstrained case to enable a better visualization of the optimization landscape. In the unconstrained case, \eqref{eq: ode3 change variable} is reduced to 
\begin{equation}   \label{eq: ode2 change variable}
   \dot{e}(t)=-\frac{1}{\alpha}\nabla_x f(e(t)+h_2(t),t) -\dot{h}_2(t).
\end{equation}

\subsubsection{Inertia encouraging the exploration} 
The first term $\nabla_x f(e+h_2(t),t)$ in \eqref{eq: ode2 change variable} can be understood as a time-varying gradient term that encourages the solution of \eqref{eq: ode2 change variable} to track $h_2(t)$, while the second term $\dot{h}_2(t)$ represents the inertia from this trajectory. In particular, if  $\dot{h}_2(t)$ points toward outside of the region of attraction of $h_2(t)$ during some time interval, the term $\dot{h}_2(t)$ acts as an \textbf{exploration} term that encourages the solution of \eqref{eq: ODE2} to leave the region of attraction of $h_2(t)$.  The parameter $\alpha$ balances the roles of the gradient and the inertia.

In the extreme case where $\alpha$ goes to infinity, $e(t)$ converges to $-h_2(t)$ and $x(t)$ approaches a constant trajectory determined by the initial point $x_0$; when $\alpha$ is sufficiently small, the time-varying gradient term dominates the inertia term and the solution of \eqref{eq: ODE2} would track $h_2(t)$ closely. With an appropriate proximal regularization $\alpha$ that keeps the balance between the time-varying gradient term and the inertia term, the solution of \eqref{eq: ODE2} could temporarily track a local minimum trajectory with the potential of exploring other local minimum trajectories.
\subsubsection{Inertia creating a one-point strongly convex landscape}
The differential equation \eqref{eq: ode2 change variable} can be written as
\begin{equation}   \label{eq: ode2 change variable h2}
   \dot{e}(t)=-\frac{1}{\alpha}  \nabla_e\Big( f(e(t)+h_2(t),t) + \alpha\dot{h}_2(t)^\top e(t)\Big)
\end{equation}
This can be regarded as a time-varying gradient flow system of the original objective function $f(e+h_2(t),t)$ plus a time-varying perturbation $\alpha\dot{h}_2(t)^\top e$. During some time interval $[t_1,t_2]$, the time-varying perturbation $\alpha\dot{h}_2(t)^\top e$ may enable the time-varying objective function $f(e+h_2(t),t)+\alpha\dot{h}_2(t)^\top e$ over a neighborhood of $h_1(t)$ to become {one-point strongly convexified} with respect to $h_2(t)$. Under such circumstances, the time-varying perturbation $\alpha\dot{h}_2(t)^\top e$ prompts the solution of \eqref{eq: ode2 change variable h2} starting in a neighborhood of $h_1(t)$ to move towards a neighborhood of $h_2(t)$. Before analyzing this phenomenon, we illustrate the concept in an example.

Consider again Example \ref{ex:bifurcation} and recall that $\bar{f}(x)$ has 2 local minima at $x=-2$ and $x=1$. By taking $b=5$, $h_1(t)=-2+5\sin(t)$ and $h_2(t)=1+5\sin(t)$, the differential equation \eqref{eq: ode2 change variable h2}  can be expressed as $\dot{e}(t)=-\frac{1}{\alpha}  \nabla_e \Big( \bar{f}(1+e(t)) + 5\alpha\cos(t)e(t) \Big)$.
The landscape of the new time-varying function  $\bar{f}(1+e) + 5\alpha\cos(t)e$ with the variable $e$ is shown  for two cases $\alpha=0.3$ and $\alpha=0.1$ in Figure  \ref{fig: new landscape}. The red curves are the solutions of \eqref{eq: ode2 change variable h2} starting from $e=-3$.
One can observe that when $\alpha=0.3$, the new landscape becomes one-point strongly convex around $h_2(t)$ over the whole region for some time interval, which provides \eqref{eq: ode2 change variable h2} with the opportunity of escaping from the region around $h_1(t)$ to the region around $h_2(t)$. However, when $\alpha=0.1$, there are always two locally one-point strongly convex regions around $h_1(t)$ and $h_2(t)$  and, therefore, \eqref{eq: ode2 change variable h2} fails to escape the region around $h_1(t)$.
\begin{figure*}[ht]
\centering
\subfloat[$\bar{f}(1+e) + 1.5\cos(t)e$]{\label{fig: change_variable 1.5}\includegraphics[width=0.45\linewidth]{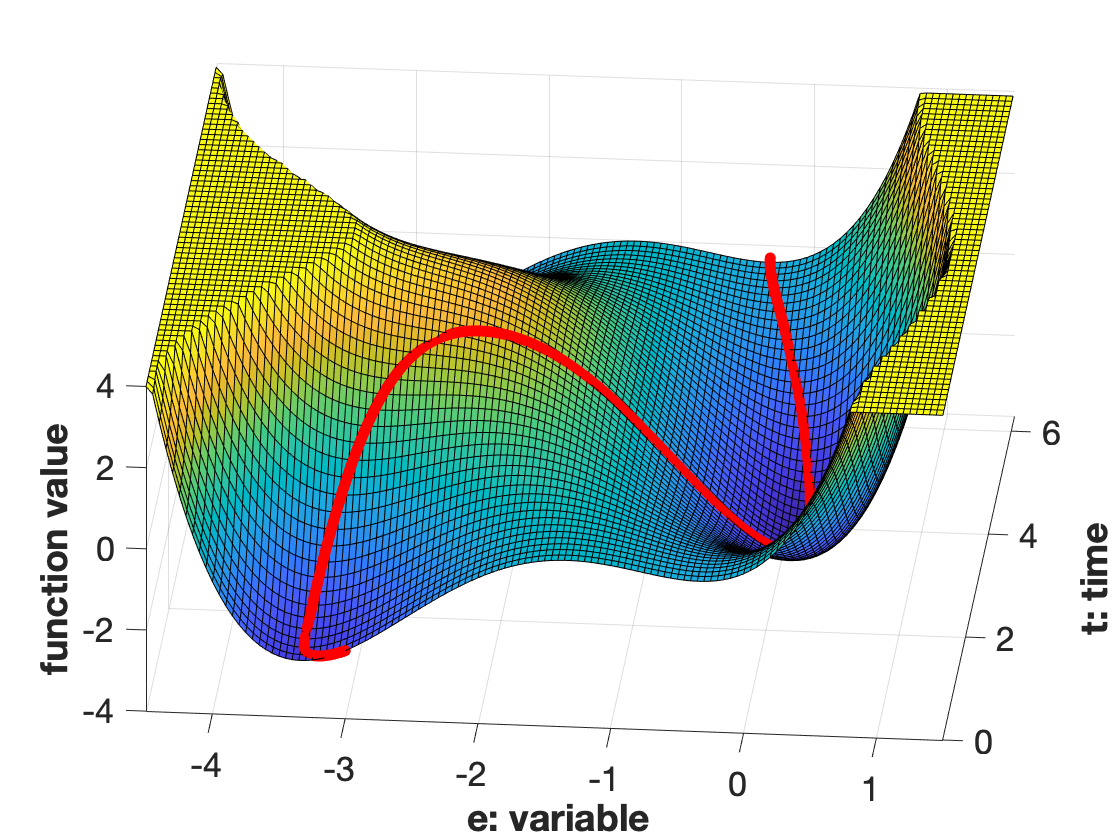}}
\subfloat[$\bar{f}(1+e) + 0.5\cos(t)e$]{\label{fig: change_variable 0.5}\includegraphics[width=0.45 \linewidth]{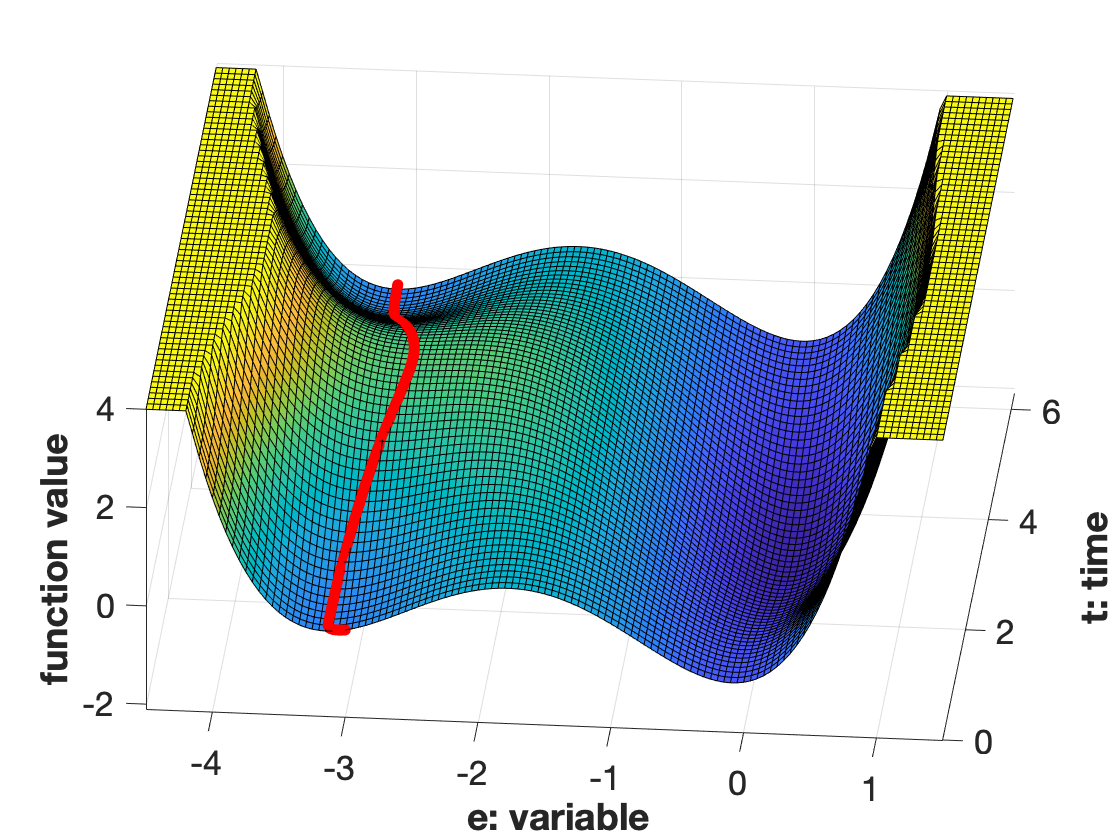}}
\caption{Illustration of time-varying landscape after change of variables for Example \ref{ex:bifurcation}.}
\label{fig: new landscape}
\end{figure*}
 
To further inspect the case $\alpha=0.3$, observe  in Figure \ref{fig: onepointconvex 5pi/6} that the landscape of the objective function $\bar{f}(1+e)+1.5\cos(0.85\pi)e$  shows that the region around the spurious local minimum trajectory $h_1(t)$ is one-point strongly convexified with respect to $h_2(t)$ at time $t=0.85\pi$. This is consistent with the fact that the solution of $\dot{e}=-\frac{1}{0.3}\nabla_x \bar{f}(1+e)-5\cos(t)$ starting from $e=-3$ jumps to the neighborhood of $0$ around time $t=0.85\pi$, as demonstrated in Figure \ref{fig: errortrack}.
\begin{figure*}[ht]
\centering
\subfloat[$\bar{f}(1+e)+1.5\cos(0.85\pi)e$]{\label{fig: onepointconvex 5pi/6}\includegraphics[width=0.32 \linewidth]{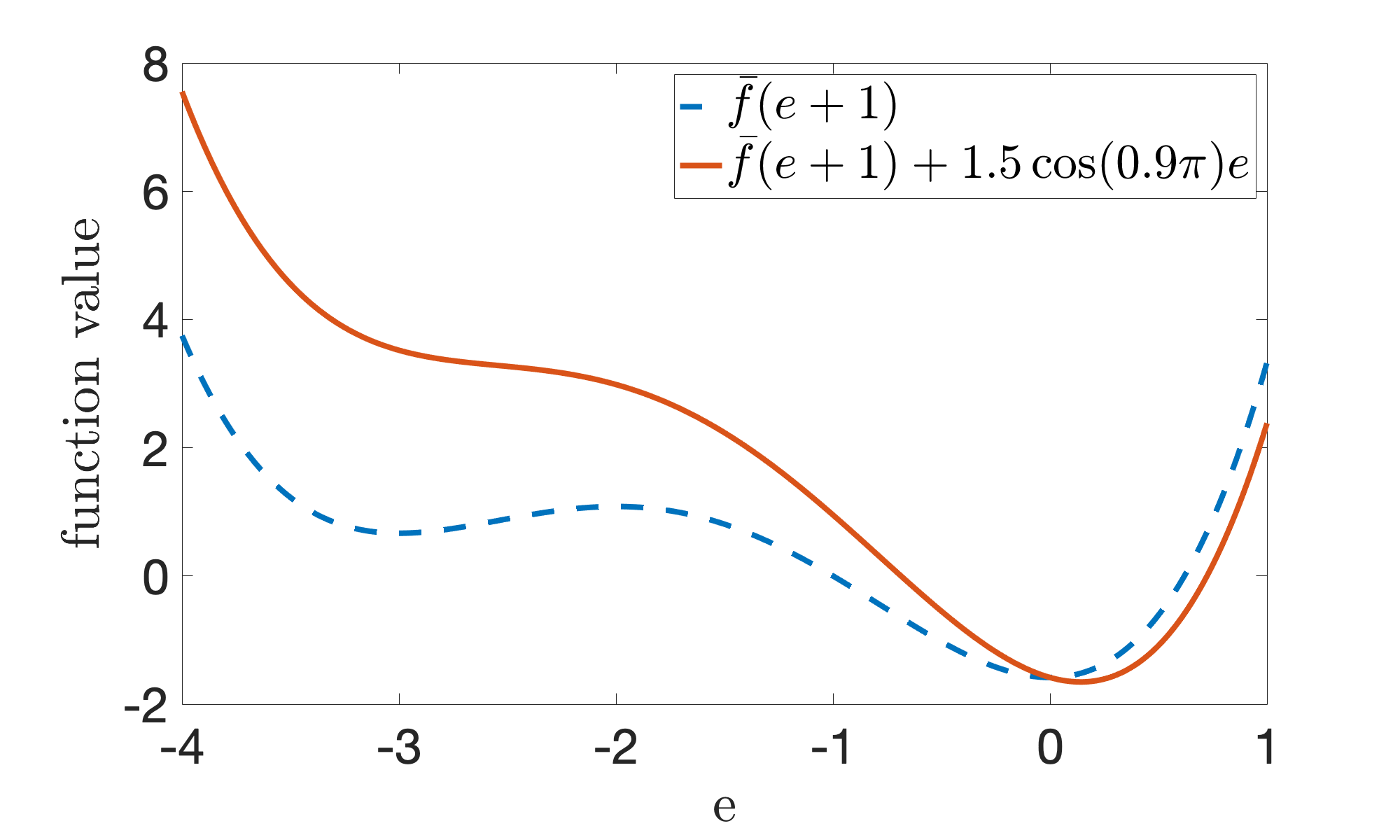}}
\subfloat[$\bar{f}(1+e)+1.5\cos(0)e$]{\label{fig: onepointconvex 0}\includegraphics[width=0.32 \linewidth]{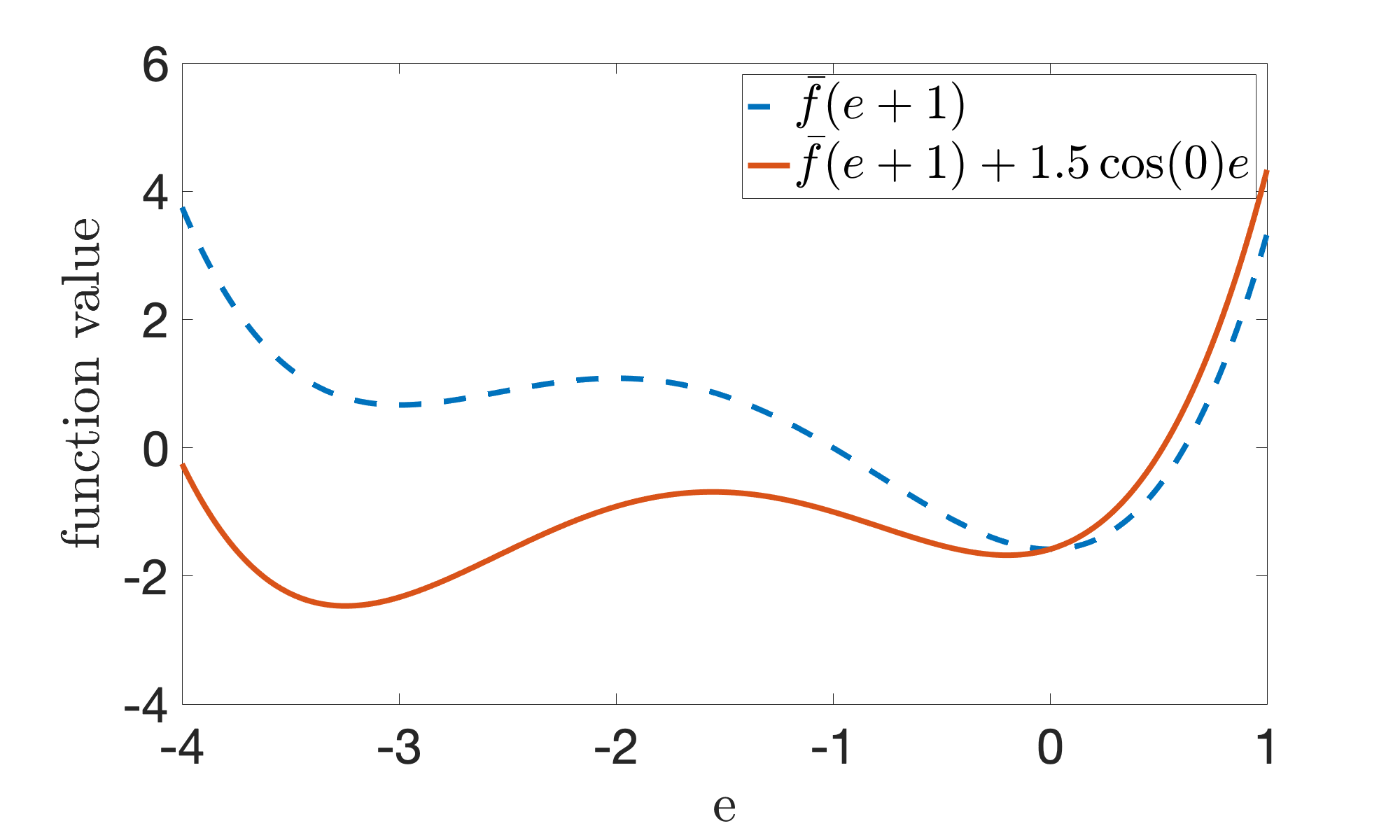}}
\subfloat[solution of $\dot{e}=-\frac{1}{0.3}\nabla_x \bar{f}(1+e)-5\cos(t)$ starting from $e_0=-3$]{\label{fig: errortrack}\includegraphics[width=0.32 \linewidth]{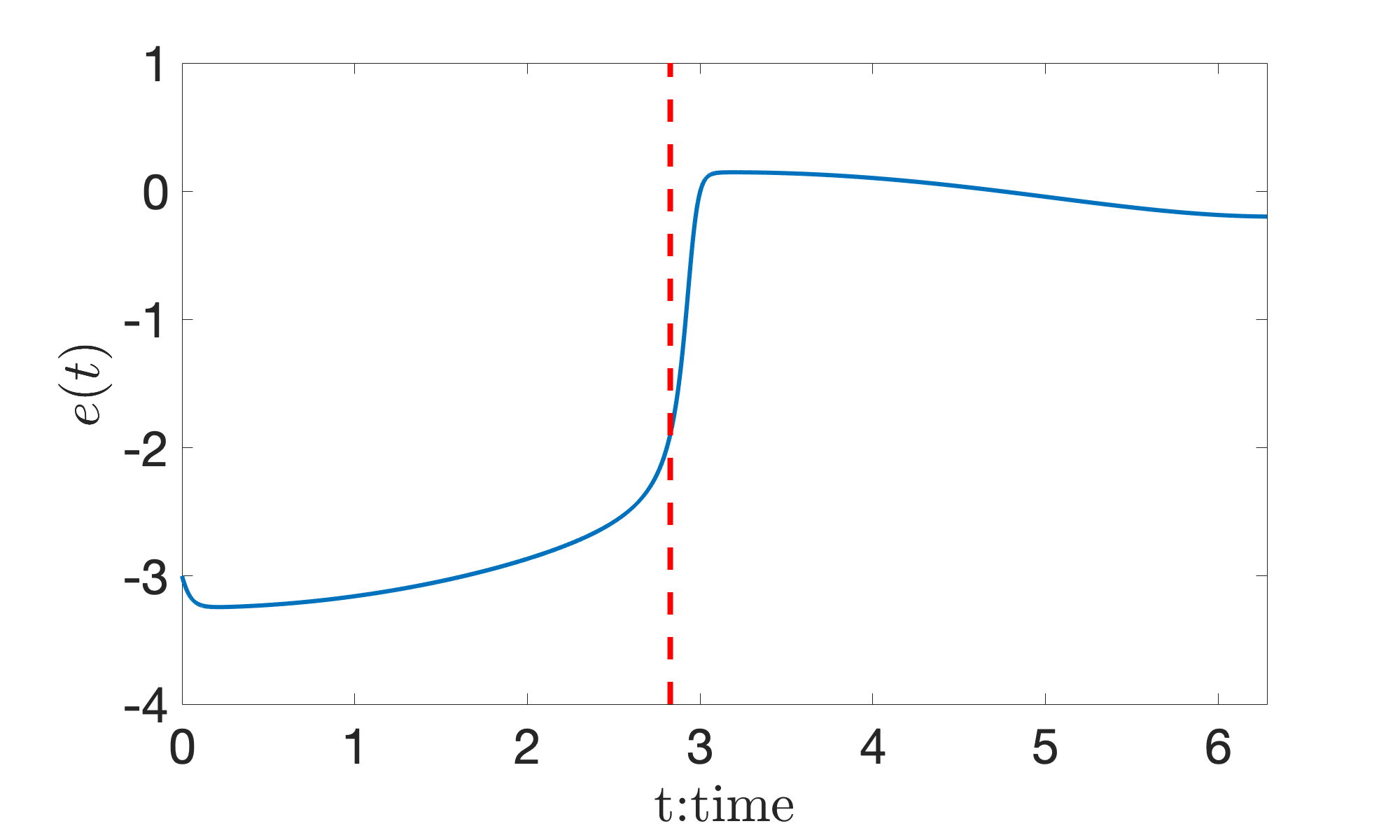}}
\caption{Illustration of one-point strong convexification for Example \ref{ex:bifurcation}.}
\label{fig: one point convex}
\end{figure*}
\begin{table*}[ht]
\centering
\caption{A unified view for unconstrained and equality-constrained problems}
\label{table:unified}
\begin{tabular}{|c|c|c|}
\hline
                                         & Unconstrained problem& Equality-constrained problem \\ \hline
 \makecell{First-order optimality\\condition(FOC) } &  $0=\nabla_x f(x,t)$  &   $0=\nabla_x L(x,\lambda,t)$  \\ \hline
 \makecell{ODE (continuous\\time limit of FOC \\for regularized problem) }
  & $\dot{x}=-\frac{1}{\alpha}\nabla_x f(x,t)$  & $\dot{x}=-\frac{1}{\alpha}\nabla_x L(x,\bar{\lambda},t)$ \\ \hline
 \makecell{Change of variables: \\ $x=h+e$}                           &    $\dot{e}=-\frac{1}{\alpha}\nabla_e f(e+h,t)-\dot{h}$            &    $\dot{e}=-\frac{1}{\alpha}\nabla_e L(e+h,\bar{\lambda},t)-\dot{h}$                   \\ \hline
 \makecell{Key assumption:\\one-point strong convexity   }           & $e^\top \nabla_e f(e+h,t)\geq c\mynorm{e}^2$              &    $e^\top \nabla_e L(e+h,\lambda,t)\geq c\mynorm{e}^2$                     \\ \hline
 \makecell{Reshaping of the landscape:\\one-point strong convexification} &  $e^\top\Big( \nabla_e f(e+h,t)+\alpha\dot{h}\Big)\geq w\mynorm{e}^2$              &   $e^\top \Big(  \nabla_e L(e+h,\bar{\lambda},t)+\alpha \dot{h} \Big)\geq w\mynorm{e}^2$                  \\ \hline
\end{tabular}
\end{table*}
Furthermore, if the time interval $[t_1,t_2]$ is large enough to allow transitioning from a neighborhood of $h_1(t)$ to a neighborhood of $h_2(t)$, then  the solution of  \eqref{eq: ode2 change variable h2} would move to the neighborhood of $h_2(t)$. In contrast, the region around $1+b\sin(t)$ is never one-point strongly convexified with respect to $-2+b\sin(t)$, as shown in Figure \ref{fig: onepointconvex 0}.

From the right-hand side of \eqref{eq: ode2 change variable h2}, it can be inferred that if the gradient of $f(\cdot,t)$ is relatively small around some local minimum trajectory, then its landscape is easier to be re-shaped by the time-varying linear perturbation $\alpha \dot{h}_2(t)^\top e$. The local minimum trajectory in a neighborhood with small gradients usually corresponds to a shallow minimum trajectory in which the trajectory has a relatively flat landscape and a relatively small region of attraction. Thus, the one-point strong convexication introduced by the time-varying perturbation could help {escape the shallow minimum trajectories}.

\subsection{\comment{Dominant trajectory}}
\comment{
In this subsection, we will formalize the intuitions discussed in Section \ref{sec: alternative view}. We first define the notion of the shallow local minimum trajectory.
\begin{definition} 
Consider a positive number $\alpha$ and assume that $\dot h_1(t)$ is $L$-Lipschitz continuous. It is said that the local minimum trajectory $h_1(t)$ is \textbf{$\alpha$-shallow} during the time period $[t_0,t_0+\delta]$ if  $\epsilon>E(\alpha)+L\delta$ and $r\leq \frac{1}{2}\delta(\epsilon-E(\alpha)-L\delta)$, where $\epsilon=\sup_{t\in [t_0,t_0+\delta]} \mynorm{\dot{h}_1(t)}$, $r=\sup_{t\in[t_0,t_0+\delta]}\sup_{x(t)\in RA^{\mathcal{M}(t)}(h_1(t))} \mynorm{x(t)-h_1(t)}$, $E(\alpha)=\sup_{t\in[t_0,t_0+\delta]}\sup_{x(t)\in RA^{\mathcal{M}(t)}(h_1(t))} \mynorm{\frac{1}{\alpha}\nabla_x L(x,\bar{\lambda},t)}$ and $\frac{1}{\alpha}\nabla_x L(x,\bar{\lambda},t)$ is defined in \eqref{eq: langange with lambda bar}.
\end{definition}}
\comment{In other words, a local minimum trajectory is shallow if it has a large time variation but a small region of attraction.
We next show that whenever a local minimum trajectory $h_1(t)$ is shallow during some time interval, the solution of \eqref{eq: ODE3} starting anywhere in the region of attraction of $h_1(t)$ will leave its region of attraction at some time.
\begin{lemma} \label{lemma: shallow local min escape}
If the local minimum trajectory $h_1(t)$ is $\alpha$-{shallow} during $[t_0,t_0+\delta]$, then for any $x(t_0) \in RA^{\mathcal{M}(t_0)}(h_1(t_0))$, then there exists a time $t\in [t_0,t_0+\delta]$ such that $x(t)\notin RA^{\mathcal{M}(t)}(h_1(t))$.
\end{lemma}
\begin{proof}
Let $b(t_0)$ be the unit vector $-\frac{ \dot{h}_1(t_0) }{ \mynorm{\dot{h}_1(t_0)} }$. One can write
\begin{align*}
    -\dot{h}_1(t)^\top b(t_0)\geq -\dot{h}_1(t_0)^\top b(t_0)-L|t-t_0|\geq \epsilon-L\delta:=\epsilon
    ^\prime
\end{align*}
For any $t\in[t_0,t_0+\delta]$ and $e(t)\in RA^{\mathcal{M}(t)}(h_1(t))$, we have
\begin{align*}
    (\dot{x}(t)-\dot{h}_1(t))^\top b(t_0)=&-\frac{1}{\alpha}\nabla_x L(x,\bar{\lambda},t)^\top b(t_0)-\dot{h}_1(t)^\top b(t_0)\\
    \geq & \epsilon^\prime-\mynorm{\frac{1}{\alpha}\nabla_x L(x,\bar{\lambda},t)} \geq \epsilon^\prime-E
\end{align*}
Hence,
\begin{align*}
    r\geq& \mynorm{x(t_0+\delta )-h_1(t_0+\delta )}\\
    \geq&  (x(t_0+\delta )-h_1(t_0+\delta ))^\top b(t_0)\\
    \geq&   (x(t_0)-h_1(t_0))^\top b(t_0) +\int_{t_0}^{t_0+\delta} (\epsilon^\prime-E) dt\\
    \geq& -r +(\epsilon^\prime-E) \delta
\end{align*}
The above contradiction completes the proof. 
\end{proof}}
\comment{
On the one hand, Lemma \ref{lemma: shallow local min escape} shows that any shallow local minimum trajectory is unstable in the sense that the time-variation in the minimum trajectory will force the solution of \eqref{eq: ODE3} to leave its region of attraction. If the shallow local minimum trajectory happens to be a non-global local solution, then the solution of \eqref{eq: ODE3}, acting as a tracking algorithm, will help avoid the bad local solutions for free. On the other hand, Lemma \ref{lemma: shallow local min escape} does not specify where the solution of \eqref{eq: ODE3} will end up after leaving the region of attraction of a shallow local minimum trajectory. Simulations (such as those provided  in Sections \ref{sec: alternative view} and \ref{sec:numercial example}) suggest that, with some appropriate $\alpha$, the solution of \eqref{eq: ODE3} may move towards a nearby local minimum trajectory that has an enlarged region of one-point strong convexity.
This leads to the following definition of the region of the domination and the {dominant} local minimum trajectory.
}
\comment{
\begin{definition} \label{def: region of domination}
Given two local minimum trajectories $h_1(t)$ and $h_2(t)$, suppose that the time-varying Lagrange function $L(x,\lambda,t)$ with $\lambda$ given in \eqref{eq: lambda} is locally $(c_2,r_2)$-one-point strongly convex with respect to $x$ around $h_2(t)$ in the region $\mathcal{M}^{h_2}(t)\cap \mathcal{B}_{r_2}(0)$. A set $D_{v,\rho,r_2}$ is said to be \textbf{the region of domination} for $h_2(t)$ with respect to $h_1(t)$ if it satisfies the following properties: 
\begin{itemize}
\item $D_{v,\rho,r_2}$ is a compact subset such that
	\begin{align} \label{eq: invariance-like set}
	e_1\in D_{v,\rho,r_2}\Rightarrow  e(t,t_1,e_1) \in D_{v,\rho,r_2},  \forall t\in [t_1,t_2]
	\end{align}
	where $e(t,t_1,e_1)$ is the solution of \eqref{eq: ode3 change variable} staring from the feasible initial point $e_1 \in {\mathcal{M}^{h_2}(t_1)}$ at the initial time $t_1$.
\item $D_{v,\rho,r_2} \supseteq	 D_v^\prime\cup\mathcal{B}_{\rho}(0)$ where
\begin{align}
\nonumber D_v^\prime=&\{e_1\in\mathbb{R}^n: e_1+h_2(t_1)\in \mathcal{M}(t_1) \cap \mathcal{B}_v(h_1(t_1)) \\
\nonumber\subseteq & RA^{\mathcal{M}(t_1)}(h_1(t_1))\},\\
\rho\geq& \sup_{t\in[t_1,t_2]}\sup_{\substack{\bar{e}(t):\mynorm{\bar{e}(t)}<r_2,\\ 0=U(\bar{e}(t),t,\alpha)}} \mynorm{\bar{e}(t)} \label{eq: e bar}.
 \end{align}
\end{itemize}
\end{definition}
The condition \eqref{eq: invariance-like set} is a set invariance property, which requires that the solution of \eqref{eq: ode3 change variable} starting from an initial point in $D_{v,\rho,r_2}$ stays in $D_{v,\rho,r_2}$ during the time period $[t_1,t_2]$.
For the visualization of $D_{v,\rho,r_2}$, $\mathcal{B}_{\rho}$ and $D_v^\prime$ in Definition \ref{def: region of domination}, we consider again Example \ref{ex:bifurcation}. In Fig \ref{fig: region of domination}, the red curve corresponds to the landscape of the function $\bar{f}(1+e)+1.5\cos(0.85\pi)e$, $e=0$ corresponds to $h_2(t)$ and $e=-3$ corresponds to $h_1(t)$. $\mathcal{B}_{\rho}$ is a region around $h_2(t)$ containing all zeros of $0=U(\cdot,t,\alpha)$ during a time period around $0.85\pi$ and $D_v^\prime$ is a neighborhood around $h_1(t)$. In this example, the region of domination for $h_2(t)$ with respect to $h_1(t)$ is $D_{v,\rho,r_2}=[-4,1]$ which contains $\mathcal{B}_{\rho}$ and $D_v^\prime$ if $h_1(t)$ if it also satisfies \eqref{eq: invariance-like set}.
\begin{figure}[ht]
\centering
  \includegraphics[width=1 \linewidth]{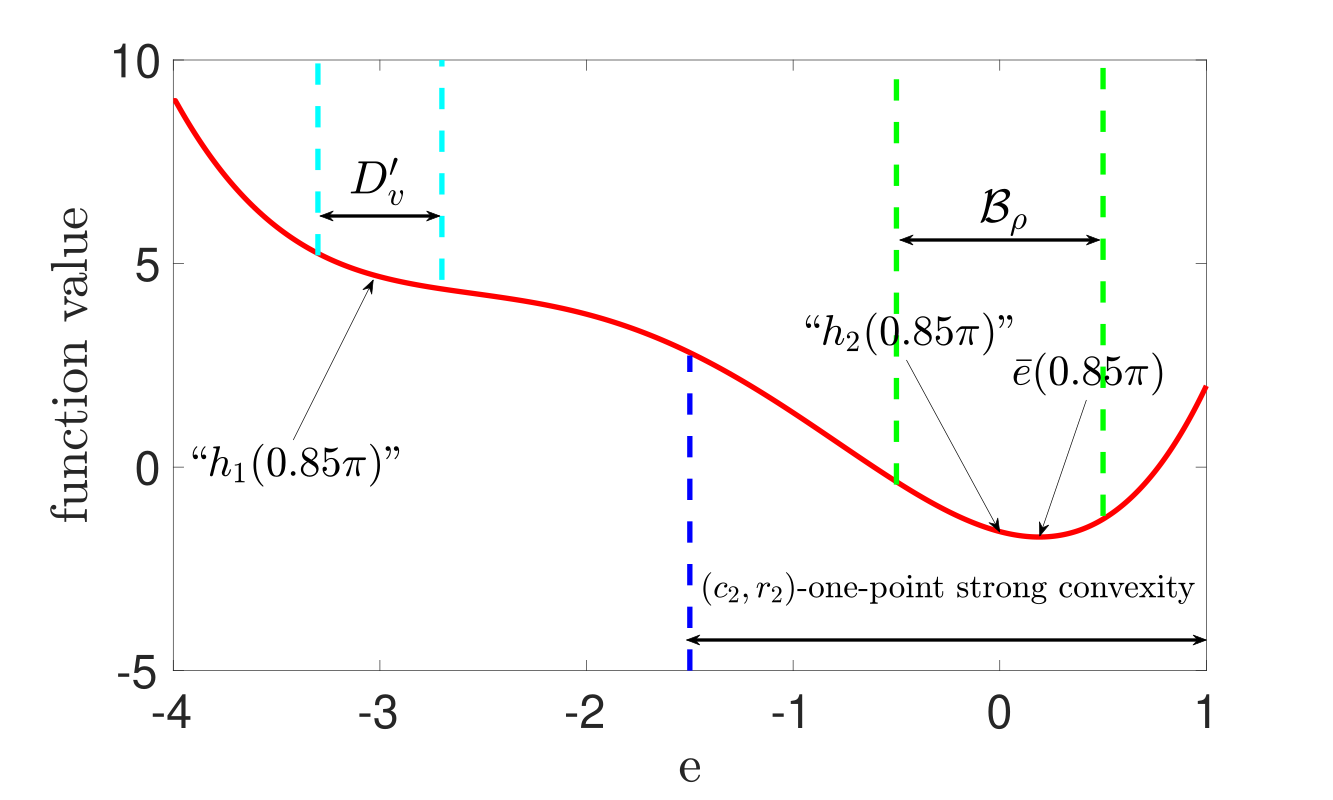}
  \caption{Illustration of Definition \ref{def: region of domination}: the region of domination.}
  \label{fig: region of domination}
\end{figure}
}
\comment{\begin{definition} \label{def: dominant trajectory}
 It is said that $h_2(t)$ is a $(\alpha,w)$-\textbf{dominant trajectory} with respect to $h_1(t)$ during the time period $[t_1,t_2]$ over the region $D_{v,\rho,r_2}$ if the time variation of $h_2(t)$ makes the time-varying function $U(e(t),t,\alpha)$ become one-point strongly monotone over $D_{v,\rho,r_2}$, i.e.,
	\begin{align} \label{eq: uniform one-point convex [t_1,t_2]}
	\nonumber	&U(e(t),t,\alpha)^\top \Big(e(t)-\bar{e}(t)\Big) \geq w\mynorm{e(t)-\bar{e}(t)}^2, \\& \forall e(t) \in D_{v,\rho,r_2}\cap{\mathcal{M}(t)},  t\in[t_1,t_2],
	\end{align}
	where $w>0$ is a constant and $\bar{e}(t)$ is defined in \eqref{eq: e bar}.
\end{definition}}
\comment{Note that $h_2(t)$ being a dominant trajectory with respect to $h_1(t)$ is equivalent to the statement that the inertia of $h_2(t)$ creates a strongly convex landscape over $D_{v,\rho,r_2}$, as discussed in Section \ref{sec: alternative view}. }

\subsection{The role of temporal variations of the constraints} \label{sec: constrained alternative view}
From the perspective of the landscape of the Lagrange functional, \eqref{eq: ode3 change variable landscape} can be regarded as a time-varying gradient flow system of the Lagrange functional $L\big(e(t)+h_2(t),\bar{\lambda}(e(t)+h_2(t),t,\alpha),t\big)$ plus a linear time-varying perturbation $\alpha\dot{h}_2^g(t)\comment{^\top}e(t)$.  Besides the linear time-varying perturbation $\alpha\dot{h}_2^g(t)\comment{^\top}e(t)$ induced by the inertia of the minimum trajectory similar to the unconstrained case, the constraints' temporal variation $g^\prime (\cdot,t)$ plays the role of shifting the Lagrange multiplier from $\lambda$ in \eqref{eq: lambda} to $\bar{\lambda}$ in \eqref{eq: lambda bar}, which results in a nonlinear time-varying perturbation of the landscape of the Lagrange functional. 

From the perspective of the perturbed gradient, the constraints' temporal variation $g^\prime (\cdot,t)$ perturbs the projected gradient $\mathcal{P}(\cdot,t)\nabla_x f(\cdot,t)$ in an orthogonal direction $\mathcal{Q}(\cdot,t)g^\prime(\cdot,t)$ to drive the trajectory of \eqref{eq: ode3 change variable gradient} towards satisfying the time-varying constraints.
\begin{lemma} \label{lemma: perpendicular}
At any given time $t$, the vector $\mathcal{P}(x,t)\nabla_x f(x,t)$ is orthogonal to the vector $\mathcal{Q}(x,t)g^\prime(x,t)$.
\end{lemma}
\begin{proof}
Recall that $\mathcal{P}(x,t)$ is the orthogonal projection matrix on the tangent plane of  $g(x(t),t)$ at the point $x(t)$ after the freezing time $t$. Thus, we have $\mathcal{P}(x,t)\nabla_x f(x,t) \in \comment{T_{x}^t}$. For the vector $\mathcal{Q}(x,t)g^\prime(x,t)$, it can be shown that
\begin{equation*}
    \mathcal{P}(x,t) \mathcal{Q}(x,t)g^\prime(x,t)=0
\end{equation*}
This implies that the orthogonal projection of the vector $\mathcal{Q}(x,t)g^\prime(x,t)$ onto the tangent plane $\comment{T_{x}^t}$ is $0$. Thus, $\mathcal{Q}(x,t)g^\prime(x,t)$ must be orthogonal to $\comment{T_{x}^t}$.
\end{proof}
Therefore, in the equality-constrained problem, the time-varying projected gradient flow system after a change of variables in \eqref{eq: ode3 change variable gradient} can be regarded as a composition of a time-varying projected term $\mathcal{P}(e+h_2(t),t)\nabla_x f(e+h_2(t),t)$, a time-varying constraint-driven term $\mathcal{Q}(e+h_2(t),t)g^\prime(e+h_2(t),t)$ and an inertia term $\dot{h}_2(t)$ due to the time variation of the local minimum trajectory.

\subsection{A unified view for unconstrained and equality-constrained problems}
By introducing the Lagrange functional in \eqref{eq: lagrange with lambda} and \eqref{eq: langange with lambda bar}, we can unify the analysis of how the temporal variation and the proximal regularization help reshape the optimization landscape and potentially make the landscape become one-point strongly convex over a larger region, for both unconstrained and equality constrained problems. This unified view is illustrated in Table \ref{table:unified}. 

\section{Main results} \label{sec: main results} 
In this section, we study the jumping, tracking and escaping properties for the time-varying nonconvex optimization.
\comment{
\subsection{Jumping}
\comment{The following theorem shows that the solution of \eqref{eq: ODE3} could jump to the dominant trajectory as long as the time-interval of such domination is large enough.}
\begin{theorem}[Sufficient conditions for jumping from $h_1(t)$ to $h_2(t)$] \label{thm: jump}
Suppose that the local minimum trajectory $h_2(t)$ is a {$(\alpha,w)$-dominant trajectory} with respect to $h_1(t)$ during $[t_1,t_2]$ over the region $D_{v,\rho,r_2}$. Let $e_1 \in D_v^\prime$ be the initial point of \eqref{eq: ode3 change variable}, and consider $\bar e(t)$ defined in \eqref{eq: e bar}.
Assume that $U(e,t,\alpha)$ is non-singular for all $t\in [t_1,t_2]$ and $e\in D_{v,\rho,r_2}$
and there exists a constant $\theta \in (0,1)$ such that
\begin{equation} \label{eq: t_2-t_1 uniform}
    t_2-t_1 \geq \max \left\{ \frac{\alpha \rho}{(r_2-\rho)\theta w},\frac{\alpha \ln \left(\frac{\mynorm{e_1-\bar{e}(t_1)}}{r_2-\rho} \right)}{(1-\theta)w}  \right\}.
\end{equation}
Then, the solution of \eqref{eq: ODE3} will $(v,r_2 )$-jump from $h_1(t)$ to $h_2(t)$ over the time interval $[t_1, t_2]$.
\end{theorem}}
\begin{proof}
\comment{First, notice that if $U(e,t,\alpha)$ is uniformly non-singular for all $t\in [t_1,t_2]$ and $e\in D_{v,\rho,r_2}$, then $\bar{e}(t)$ defined in \eqref{eq: e bar} is continuously differentiable for $t\in[t_1,t_2]$.}
Then, notice that every solution of \eqref{eq: ode3 change variable} with an initial point in \comment{$D_{v,\rho,r_2}\cap\mathcal{M}(t_1)$} will remain in \comment{$D_{v,\rho,r_2}$}. It follows from Theorem \ref{lemma: exist and unique} that \eqref{eq: ode3 change variable} has a unique solution defined for all $t\in[t_1,t_2]$ whenever $e_1\in\comment{ D_{v,\rho,r_2}}\cap \mathcal{M}(t_1)$.

We take $V(e(t),t)=\frac{1}{2}\mynorm{e(t)-\bar{e}(t)}^2$ as the Lyapunov function for the system \eqref{eq: ode3 change variable}. Because of Lemma \ref{lemma: invariance}, any solution of  \eqref{eq: ode3 change variable} stating in $\mathcal{M}(t_1)$ will remain in $\mathcal{M}(t)$ for all $t\geq t_1$. Therefore, the derivative of $V(e(t),t)$ along the trajectories of \eqref{eq: ode3 change variable} in  $\mathcal{M}(t)$ can be expressed as 
\comment{
\begin{align}
\nonumber\dot{V}=&(e(t)-\bar{e}(t))^\top \Big(-\frac{1}{\alpha}U(e(t),t,\alpha) \Big) -\\ \nonumber&(e(t)-\bar{e}(t))^\top \dot{\bar{e}}(t),
\ \forall e(t) \in D_{v,\rho,r_2}\cap\mathcal{M}^{h_2}(t)\\
\nonumber\leq& -\frac{w}{\alpha}\mynorm{e(t)-\bar{e}(t)}^2 +\mynorm{\dot{\bar{e}}(t)}\mynorm{e(t)-\bar{e}(t)}, \\
\nonumber&\forall e(t) \in D_{v,\rho,r_2}\cap\mathcal{M}^{h_2}(t)\\
\nonumber\leq&-(1-\theta)\frac{w}{\alpha}\mynorm{e(t)-\bar{e}(t)}^2 -\theta\frac{w}{\alpha}\mynorm{e(t)-\bar{e}(t)}^2 \\
\nonumber&+\delta\mynorm{e(t)-\bar{e}(t)}, \forall e(t) \in D_{v,\rho,r_2}\cap\mathcal{M}^{h_2}(t)\\
\nonumber\leq& -(1-\theta)\frac{w}{\alpha}\mynorm{e(t)-\bar{e}(t)}^2,\forall e(t) \in \\ 
&  \{e(t) \in D_{v,\rho,r_2}\cap\mathcal{M}^{h_2}(t): \mynorm{e(t)-\bar{e}(t)}\geq \frac{\alpha \delta}{\theta w}\}  \label{eq: jump dot V along ODE3} 
\end{align}
where $\delta\coloneqq\sup_{t\in[t_1,t_2]}\mynorm{\dot{\bar{e}}(t)}$.  By taking $e_1\in \comment{D_v^\prime \cap\mathcal{M}(t_1)}$, since \comment{$D_{v,\rho,r_2}$ satisfies the condition \eqref{eq: invariance-like set}},  the solution of \eqref{eq: ode3 change variable} starting from $e_1$ will stay in \comment{$D_{v,\rho,r_2}$}. Thus, the bound in \eqref{eq: jump dot V along ODE3} is valid.
To ensure that the trajectory of \eqref{eq: ode3 change variable}  enters the time-varying set \comment{$\mathcal{B}_{r_2-\rho}(\bar{e}(t))$}, it is sufficient to have $\frac{\alpha \delta}{\theta w}\leq r_2-\rho$ or $\alpha \leq \frac{(r_2-\rho)\theta w}{\delta}$. Since $\delta=\sup_{t\in[t_1,t_2]}\mynorm{\dot{\bar{e}}(t)}\geq \frac{\rho}{t_2-t_1}$. We can further bound $\alpha$ as $\alpha\leq \frac{(r_2-\rho)\theta w (t_2-t_1)}{\rho} $ which is equivalent to $t_2-t_1\geq \frac{\alpha \rho}{(r_2-\rho)\theta w}$.}

Now, it is desirable to show that if the time interval $[t_1,t_2]$ is large enough, the solution of \eqref{eq: ode3 change variable gradient} will enter the time-varying set $\mathcal{B}_{r_2-\rho}(\bar{e}(t))$ with an exponential convergence rate.
Since $\dot{V}(\cdot,\cdot)$ is negative in $\Gamma(t)\coloneqq\{e\in \comment{D_{v,\rho,r_2}}\cap\mathcal{M}^{h_2}(t): \mynorm{e-\bar{e}(t)}\geq \frac{\alpha \delta}{\theta w}\}$ and because of \eqref{eq: invariance-like set}, a trajectory starting from $\Gamma(t_1)$ must stay in \comment{$D_{v,\rho,r_2}$} and move in a direction of decreasing $V(e,t)$. The function $V(e,t)$ will continue decreasing until the trajectory enters the set $\{e\in \comment{D_{v,\rho,r_2}}\cap\mathcal{M}^{h_2}(t): \mynorm{e-\bar{e}(t)}\leq \frac{\alpha \delta}{\theta w}\}$ or until time $t_2$. Let us show that the trajectory enters $\mathcal{B}_{r_2-\rho}(\bar{e}(t))$ before $t_2$ if $t_2-t_1>\frac{\alpha}{w(1-\theta)} \ln(\frac{\mynorm{e_1-\bar{e}(t_1)}}{r_2-\rho})$.
Since $V(e(t),t)=\frac{1}{2}\mynorm{e(t)-\bar{e}(t)}^2$, \eqref{eq: jump dot V along ODE3} can be written as
\begin{align*}
&\dot{V}(e(t),t) \leq -(1-\theta)\frac{2w}{\alpha} V(e(t),t), \\
&\forall e \in \Big\{e \in \comment{D_{v,\rho,r_2}}\cap\mathcal{M}^{h_2}(t): \mynorm{e(t)-\bar{e}(t)} \geq \frac{\alpha \delta}{\theta w}\} \Big\},
\end{align*}
By the comparison lemma\cite[Lemma 3.4]{khalil2002nonlinear},
\begin{align*}
V(e(t),t) \leq \exp\comment{\{}-(1-\theta)\frac{2w}{\alpha}(t-t_1)\comment{\}} V(e_1,t_1)
\end{align*}
Hence,
\begin{align*}
\mynorm{e(t)-\bar{e}(t)}\leq \exp\comment{\{}-(1-\theta)\frac{w}{\alpha}(t-t_1)\comment{\}} \mynorm{e_1-\bar{e}(t_1)}.
\end{align*}
The inequality  $\mynorm{e(t_2)-\bar{e}(t_2)}\leq  r_2-\rho$ holds if $t_2-t_1\geq \frac{\alpha}{w(1-\theta)} \ln(\frac{\mynorm{e_1-\bar{e}(t_1)}}{r_2-\rho})$.
\end{proof}

We also offer an approach based on the time-averaged dynamics over a small time interval and name it ``small interval averaging"\footnote{Our averaging approach distinguishes from classic averaging methods \cite{hale1980ODE,khalil2002nonlinear,teel1999semi,aeyels1999exponential} and the partial averaging method \cite{peuteman2002exponential} in the sense that: (1) it is averaged over a small time interval instead of the entire time horizon, and (2) there is no two-time-scale behavior because there is no parameter in \eqref{eq: ode2 change variable} that can be taken sufficiently small.}.
This technique guarantees that the solution of the  time-varying differential equation (or system)  will converge to a residual set of the origin of \eqref{eq: ode2 change variable}, provided that: (i) there is a time interval $[t_1,t_2]$ such that the temporal variation makes the averaged objective function during this interval  locally one-point strongly convex around $h_2(t)$ not only just over a neighborhood of $h_2(t)$  but also over a  neighborhood of $h_1(t)$,
(ii) the original time-varying system is not too distant from the time-invariant averaged system,
(iii) $[t_1,t_2]$ is large enough to allow the transition of points from a neighborhood of $h_1(t)$ to a neighborhood of $h_2(t)$.
Therefore, the time interval $[t_1,t_2]$ and the time-averaged dynamics over this time interval serve as  a certificate for jumping from $h_1(t)$ to $h_2(t)$. In what follows, we introduce the notion of averaging a time-varying function over a time interval $[t_1,t_2]$.
\comment{
\begin{definition} \label{def: interval average}
A function $U_{\text{av}}(e,\alpha)$ is said to be the \textbf{average function} of $U(e,t,\alpha)$ over the time interval $[t_1,t_2]$ if 
\begin{equation*} 
  U_{\text{av}}(e,\alpha)= \frac{1}{t_2-t_1} \int_{t_1}^{t_2} U(e,t,\alpha) d\tau
\end{equation*}
\end{definition}
The averaged system of  \eqref{eq: ode3 change variable} over the time interval $[t_1,t_2]$ can be written as 
\begin{equation}   \label{eq: partial average system constrained}
\dot{e}=-\frac{1}{\alpha}U_{\text{av}}(e,\alpha)
\end{equation}
Then, \eqref{eq: ode3 change variable} can be regarded as a time-invariant system \eqref{eq: partial average system constrained} with the time-varying perturbation term $p(e(t),t,\alpha)=-\frac{1}{\alpha}(U(e(t),t,\alpha)-U_{\text{av}}(e,\alpha))$.}
\comment{For the averaged system, we can define the on-average region of domination $D_{v,\rho,r_2}$ for $h_2(t)$ with respect to $h_1(t)$ similarly as Definition \ref{def: region of domination} by replacing \eqref{eq: e bar} with 
\begin{equation}\label{eq: e bar average}
    \rho\geq \sup_{\substack{\bar{e}:\mynorm{\bar{e}}<r_2, 0=U_{av}(\bar{e},\alpha)}} \mynorm{\bar{e}}.
\end{equation}
The corresponding on-average $(\alpha,w)$-{dominant trajectory} with respect to $h_1(t)$ during $[t_1,t_2]$ over the region $D_{v,\rho,r_2}$ can also be defined similarly as Definition \ref{def: dominant trajectory} by replacing \eqref{eq: uniform one-point convex [t_1,t_2]} with 
\begin{align} \label{eq: strong convex for jump average}
 &U_{av}(e,\alpha)^\top (e-\bar{e}) \geq w\mynorm{e-\bar{e}}^2, \\
\nonumber &\forall e \in D_{v,\rho,r_2}\cup \left(\cup_{[t_1,t_2]}\mathcal{M}(t)\right)
\end{align}
where $\bar{e}$ is defined in \eqref{eq: e bar average}.
}
\comment{
\begin{theorem}[Sufficient conditions for jumping from $h_1(t)$ to $h_2(t)$ using averaging]  \label{thm: jump average}
Suppose that the local minimum trajectory $h_2(t)$ is a on-average {$(\alpha,w)$-dominant trajectory} with respect to $h_1(t)$ during $[t_1,t_2]$ over the region $D_{v,\rho,r_2}$. Assume that the following conditions are satisfied:
\begin{enumerate}
	    \item    {There exist some time-varying scalar functions $\delta_1(\alpha,t)$ and $\delta_2(\alpha,t)$ such that}
	    \begin{align} \label{eq: perturbation 1 jumping}
    	\mynorm{p(e(t),t,\alpha)}\leq \delta_1(\alpha,t) \mynorm{e-\bar{e}} + \delta_2(\alpha,t),
    \end{align}
    for all $t\in [t_1,t_2]$, and there exist some positive constants $\eta_1(\alpha)$ and $\eta_2(\alpha)$ such that
    \begin{align} \label{eq: perturbation 2 jumping}
    	\int_{t_1}^{t}\delta_1(\alpha,\tau)d\tau \leq \eta_1(\alpha)(t-t_1)+\eta_2(\alpha).
    \end{align}
    \item  The inequality
    \begin{align} \label{eq: t_2-t_1}
& \beta_2(\alpha) \mynorm{e_1-\bar{e}} e^{-\beta_1(\alpha)(t_2-t_1)} +{ \beta_2(\alpha)}\\
\nonumber&\int_{t_1}^{t_2}e^{-\beta_1(\alpha)(t_2-\tau)}\delta_2(\alpha, \tau)d\tau  \leq  r_2-\rho, \forall e_1 \in D_v^\prime
    \end{align}
holds, where $\beta_1(\alpha)= \frac{w}{\alpha}-\eta_1(\alpha)>0$ and $\beta_2(\alpha)=e^{\eta_2(\alpha)}\geq 1$.
\end{enumerate}
Then, the solution of \eqref{eq: ODE3} will $(v,r_2 )$-jump from $h_1(t)$ to $h_2(t)$ over the time interval $[t_1, t_2]$.
\end{theorem}}
\begin{proof}
As shown in the proof of Theorem \ref{thm: jump}, the differential equation \eqref{eq: ode3 change variable} has a unique solution defined for all $t\in[t_1,t_2]$ that stays in $\mathcal{M}^{h_2}(t)$ whenever $e_1\in \comment{D_{v,\rho,r_2}}\cap \mathcal{M}^{h_2}(t_1)$.
By using $V(e)=\frac{1}{2}\mynorm{e-\bar{e}}^2: \comment{D_{v,\rho,r_2}} \rightarrow \mathbb{R}$ as the Lyapunov function for the system \eqref{eq: ode3 change variable},  the derivative of $V(e)$ along the trajectories of \eqref{eq: ode3 change variable} can be obtained~as 
\begin{align*}
\nonumber\dot{V}(e)&=(e-\bar{e})^\top \comment{\Big( -\frac{1}{\alpha}U_{av}(e,\alpha) +p(e,\alpha,t)\Big)}\\
&\leq -\frac{w}{\alpha}\mynorm{e-\bar{e}}^2 + \delta_1(\alpha,t)\mynorm{e-\bar{e}}^2+\delta_2(\alpha,t)\mynorm{e-\bar{e}}
\end{align*}
Since $V(e)=\frac{1}{2}\mynorm{e-\bar{e}}^2$, one can derive an upper bound on $\dot{V}$ as
\begin{equation*}
    \dot{V}(e) \leq -\Big[ \frac{2w}{\alpha} - 2\delta_1(\alpha,t)\Big]V(e) + \delta_2(\alpha,t) \sqrt{2V(e)}
\end{equation*}
To obtain a linear differential inequality, we consider $W(t)=\sqrt{V(e(t))}$.
When $V(e(t))\neq 0$, it holds that $\dot{W}=\dot{V}/2\sqrt{V}$ and 
\begin{equation} \label{eq: W inequailty}
    \dot{W}\leq -\Big[ \frac{w}{\alpha} - \delta_1(\alpha,t)\Big]W + \frac{\delta_2(\alpha,t)}{\sqrt{2}}
\end{equation}
When $V(e(t))=0$, we have $e(t)=\bar{e}$. Writing the Tylor expansion of $e(t+\epsilon)$ for a sufficiently small $\epsilon$ yields that
\comment{
\begin{align*}
 e(t+\epsilon)=&e(t)+\epsilon\Big( -\frac{1}{\alpha}U_{av}(e,\alpha) +p(e,\alpha,t)\Big) +o(\epsilon)\\
 =&\bar{e}+\epsilon p(\bar{e},\alpha,t) +o(\epsilon)
\end{align*}}
This implies that
 \begin{align*}
  V(e(t+\epsilon))=\frac{\epsilon^2}{2}\mynorm{\comment{p(\bar{e},\alpha,t)}}^2 + o(\epsilon^2).
 \end{align*}
Therefore,
\begin{equation} \label{eq: upper right hand of W}
    \begin{aligned}
    D^+W(t)&=\lim \sup _{\epsilon \rightarrow 0^+} \frac{W(t+\epsilon)-W(t)}{\epsilon}\\
     &=\lim \sup _{\epsilon \rightarrow 0^+} \frac{\sqrt{\frac{\epsilon^2}{2}\mynorm{\comment{p(\bar{e},\alpha,t)}}^2 + o(\epsilon^2)}}{\epsilon} \\
   &= \frac{1}{\sqrt{2} }\mynorm{\comment{p(\bar{e},\alpha,t)}} \\
      &\leq \frac{1}{\sqrt{2} } \delta_2 (\alpha,t)
    \end{aligned}
\end{equation}
Thus, \eqref{eq: W inequailty} is also satisfied when $V=0$, and accordingly $D^+W(t)$ satisfies \eqref{eq: W inequailty} for all values of $V$. Since $W$ is scalar and the right-hand side of \eqref{eq: W inequailty} is continuous in $t$ and locally Lipschitz in $W$  for all $t\in [t_1,t_2]$ and  $W\geq 0$, the comparison lemma is applicable. In addition, the right-hand side of \eqref{eq: W inequailty} is linear and a closed-form expression for the solution of the first-order linear differential equation of $W$ can be obtained. Hence, $W(t)$ satisfies 
\begin{equation} \label{eq: W(t)}
    W(t)\leq \phi(t,t_1)W(t_1)+ \frac{1}{\sqrt{2}} \int_{t_1}^t \phi(t,\tau)\delta_2(\alpha,\tau) d\tau
\end{equation}
where the translation function $\phi(t,t_1)$ is given by 
\begin{equation} 
    \phi(t,t_1)=\exp\Big[ -\frac{w}{\alpha }(t-t_1)+\int_{t_1}^t\delta_1(\alpha,\tau)d\tau \big].
\end{equation}
\begin{equation} \label{eq: bound on x}
    \mynorm{e(t)-\bar{e}}\leq  \phi(t,t_1)\mynorm{e_1-\bar{e}}+\int_{t_1}^t \phi(t,\tau)\delta_2(\alpha,\tau)d\tau
\end{equation}
Since $\int_{t_1}^t\delta_1(\alpha,\tau)d\tau\leq \eta_1(\alpha)  (t-t_1)+ \eta_2(\alpha)$, and using $\beta_1(\alpha)= \frac{w}{\alpha}-\eta_1(\alpha)>0$ and $ \beta_2(\alpha)=e^{\eta_2(\alpha)}\geq 1$ in \eqref{eq: bound on x}, it holds that
\begin{align} \label{eq:bound in averaging method constrained}
\nonumber  \mynorm{e(t)-\bar{e}}\leq& \beta_2(\alpha)\mynorm{e_1-\bar{e}}e^{-\beta_1(\alpha) (t-t_1)}\\
  &+\beta_2(\alpha)\int_{t_1}^t e^{-\beta_1(\alpha) (t-\tau)}\delta_2(\alpha,\tau)d\tau
\end{align}
By taking $e_1 \in \comment{D_v^\prime \subseteq D_{v,\rho,r_2}}$, since $D_{v,\rho,r_2}$ retains trajectories starting from a feasible initial point with respect to the dynamics \eqref{eq: ode3 change variable} for $t\in[t_1,t_2]$,  any trajectory of \eqref{eq: ode3 change variable} starting from $D_v^\prime$ will stay in $\comment{D_{v,\rho,r_2}}$ and remain in the feasible set $\mathcal{M}^{h_2}(t)$. Thus, the bound in \eqref{eq:bound in averaging method constrained} is valid.
If $t_2$ satisfies 
\begin{align*}
&\beta_2(\alpha)\mynorm{e_1-\bar{e}}e^{-\beta_1(\alpha) (t_2-t_1)}\\
&+\beta_2(\alpha)\int_{t_1}^{t_2} e^{-\beta_1(\alpha) (t_2-\tau)}\delta_2(\alpha,\tau)d\tau \leq r_2-\rho
\end{align*}
then $ \mynorm{e(t_2)-\bar{e}}\leq r_2-\rho$. Since $\bar{e} \in \comment{\mathcal{B}_{\rho}(0)}$, we have $ \mynorm{e(t_2)}\leq r_2$. This shows that the solution of \eqref{eq: ode2 change variable} jumps from $h_1(t)$ to $h_2(t)$ during the time interval $[t_1,t_2]$. 
\end{proof}

\comment{
\begin{remark}
If the global minimum trajectory is the dominant trajectory with respect to the spurious local minimum trajectories, then Theorems \ref{thm: jump} and \ref{thm: jump average} guarantee that the solution of \eqref{eq: ODE3} will jump to the neighborhood of the global minimum trajectory.
\end{remark}
\begin{remark}
The condition in Theorem \ref{thm: jump} and Condition 2 in Theorem \ref{thm: jump average} mean that $[t_1,t_2]$ needs to be large enough to allow the transition of points from a neighborhood of $h_1(t)$ to a neighborhood of $h_2(t)$.
Condition 1 in Theorem \ref{thm: jump average} means that the original time-varying system should not be too distant from the time-invariant averaged system.
\end{remark}
\begin{remark}
To make the one-point strong monotonicity conditions \eqref{eq: uniform one-point convex [t_1,t_2]} and \eqref{eq: strong convex for jump average} hold, the inertia parameter $\alpha$ cannot be too small.
\end{remark}}
\begin{remark}
The locally one-point strongly convex parameter $w$ in \eqref{eq: uniform one-point convex [t_1,t_2]} and \eqref{eq: strong convex for jump average} determines the convergence rate during $[t_1,t_2]$, which is reflected in \eqref{eq: t_2-t_1 uniform} and \eqref{eq: t_2-t_1}.
\end{remark}
\begin{remark}
In Theorem \ref{thm: jump average}, to ensure that the time-invariant partial interval averaged system is a reasonable approximation of the time-varying system, the time interval $[t_1,t_2]$ should not be very large. On the other hand, to guarantee that the solution of \eqref{eq: ode3 change variable} has enough time to jump, the time interval $[t_1,t_2]$ should not be very small. This trade-off is reflected in \eqref{eq: t_2-t_1}.
\end{remark}

\subsection{Tracking}
In this subsection, we study the tracking property of the local minimum trajectory $h_2(t)$. First, notice that if $h_2(t)$ is not constant,  the right-hand side of \eqref{eq: ODE3} is nonzero while the left-hand side is zero. Therefore, $h_2(t)$ is not a solution of \eqref{eq: ODE3} in general. This is because the solution of \eqref{eq: ODE3} approximates the continuous limit of a discrete local trajectory of the sequential {regularized} optimization problem \eqref{eq: regularized constrained opt}.  However, to preserve the optimality of the solution with regards to  the original time-varying optimization problem without any proximal regularization, it is required to guarantee that the solution of \eqref{eq: ODE3} is close to $h_2(t)$.

If the solution of \eqref{eq: ode3 change variable} can be shown to be in a small residual set around $0$ on the time-varying manifold $\mathcal{M}(t)$, then it is guaranteed that $x(t,t_0,x_0)$ tracks its nearby local minimum trajectory. Notice that \eqref{eq: ode3 change variable} can be regarded as a time-varying perturbation of the  system
\begin{equation}   \label{eq: no time-varying term}
   \dot{e}=-\frac{1}{\alpha}\mathcal{P}(e+h_2(t),t)\nabla_x f(e+h_2(t),t),\quad \forall t\geq t_0
\end{equation}
Since $h_2(t)$ is a local minimum trajectory, it is obvious that $e(t) \equiv 0$ is an equilibrium point of \eqref{eq: no time-varying term}. 
In addition, if the time-varying Lagrange function $L(x,\lambda,t)$ with $\lambda$ given in \eqref{eq: lambda} is {locally one-point strongly convex with respect to $x$} around $h_2(t)$ in the time-varying feasible set $\mathcal{M}(t)$, after noticing the fact that the solution of \eqref{eq: ode3 change variable} will remain in $\mathcal{M}^{h_2}(t)$ if the initial point $e_0\in\mathcal{M}^{h_2}(t_0)$ from Lemma \ref{lemma: invariance},
one would expect that the solution of \eqref{eq: ode3 change variable} stays in a small residual set of $e=0$ if the perturbation $\mathcal{Q}(e(t)+h_2(t),t)g^\prime(e(t)+h_2(t),t)+\dot{h}_2(t)$ is relatively small. The perturbation  $\mathcal{Q}(e(t)+h_2(t),t)g^\prime(e(t)+h_2(t),t)+\dot{h}_2(t)$ being  small is equivalent to $\alpha$ being  small. The next theorem shows that \comment{every local minimum trajectory} can be tracked for a relatively small $\alpha$.

\begin{theorem}[Sufficient condition for tracking] \label{thm: Sufficient condition for tracking constrained}
 Assume that the time-varying Lagrange function $L(x,\lambda,t)$ with $\lambda$ given in \eqref{eq: lambda} is \comment{{locally $(c_2,r_2)$-one-point strongly convex} with respect to $x$} around $h_2(t)$. Given $\gamma(t)$ such that $\mynorm{\dot{h}_2(t)}\leq \gamma(t)$, suppose that \comment{there exist time-varying scalar functions $\delta_1(t)$ and $\delta_2(t)$} such that the perturbed gradient due to the time-variation of constraints satisfies the inequality
    \begin{ealign} \label{eq: perturbation 1}
    	\mynorm{\mathcal{Q}(e(t)+h_2(t),t)g^\prime(e(t)+h_2(t),t)}\leq \delta_1(t) \mynorm{e} + \delta_2(t),
    \end{ealign}
    and there exist some positive constants $\eta_1$ and $\eta_2$ such that
    \begin{ealign} \label{eq: perturbation 2}
    	\int_{t_1}^{t}\delta_1(\tau)d\tau \leq \eta_1(t-t_1)+\eta_2.
    \end{ealign}
  If \comment{$\sup_{t\geq {t_1}}( \delta_2 (t)+\gamma(t))$ is bounded} and the following conditions hold
  \begin{subequations} \label{eq: condition for tracking constrained}
  \begin{align}
  &\mynorm{x_0-h_2(0)}\leq \frac{r_2}{e^{\eta_2}},\\ 
  &\alpha\leq \frac{c_2r_2}{e^{\eta_2}\sup_{t\geq \comment{t_1}}( \delta_2 (t)+\gamma(t))+\eta_1r_2 }, \label{eq: tracking upper bound on a}
  \end{align}
  \end{subequations}
  then the solution $x(t,t_0,x_0)$ will {$r_2$-track} $h_2(t)$. More specifically, we have
\begin{align}\label{eq: bound for tracking constrained}
 \nonumber& \mynorm{x(t,t_0,x_0)-h_2(t)}\leq e^{\eta_2}\mynorm{e_1}e^{-(\frac{c_2}{\alpha}-\eta_1) (t-t_1)}\\
  &+e^{\eta_2}\int_{t_1}^t e^{-(\frac{c_2}{\alpha}-\eta_1) (t-\tau)}( \delta_2 (t)+\gamma(t))d\tau \leq r_2.
\end{align}
\end{theorem}
\begin{proof}
Consider $V(e)=\frac{1}{2}\mynorm{e}^2: \mathcal{B}_{r_2}(0) \rightarrow \mathbb{R}$ as the Lyapunov function for the system \eqref{eq: ode3 change variable}. Because of Lemma \ref{lemma: invariance}, any solution of  \eqref{eq: ode3 change variable} stating in ${\mathcal{M}(t_1)}$ will remain in $\mathcal{M}(t)$ for all $t\geq {t_1}$. The derivative of $V(e)$ along the trajectories of \eqref{eq: ode3 change variable} can be obtained~as 
\begin{align*}
\dot{V}&=e(t)^\top \Big(-\frac{1}{\alpha}\mathcal{P}(e(t)+h_2(t),t)\nabla_x f(e(t)+h_2(t),t)\\
&- \mathcal{Q}(e(t)+h_2(t),t)g^\prime(t)(e(t)+h_2(t),t) - \dot{h}_2^g(t)\Big),\\
&\leq -\frac{c}{\alpha}\mynorm{e(t)}^2 + \delta_1(t)\mynorm{e(t)}^2+(\delta_2(t)+\gamma(t))\mynorm{e(t)} 
\end{align*}
Since $V(e)=\frac{1}{2}\mynorm{e}^2$, one can derive an upper bound on $\dot{V}$ as
\begin{equation*}
    \dot{V} \leq -\Big[ \frac{2c}{\alpha} - 2\delta_1(t)\Big]V + (\delta_2(t)+\gamma(t)) \sqrt{2V}
\end{equation*}

Using the same proof procedure as in Theorem \ref{thm: jump average} and by taking $\beta_1(\alpha)= \frac{c}{\alpha}-\eta_1>0$ and $\beta_2=e^{\eta_2}\geq 1$, it can be shown that
\begin{ealign} \label{eq:bound in constrained tracking}
  \mynorm{e(t)}&\leq \beta_2\mynorm{e_1}e^{-\beta_1(\alpha) (t-t_1)}\\
  &+\beta_2\int_{t_1}^t e^{-\beta_1(\alpha) (t-\tau)}( \delta_2 (t)+\gamma(t))d\tau
\end{ealign}
To make the bound in \eqref{eq:bound in constrained tracking} valid, we must ensure that ${e(t)}\in\mathcal{B}_{r_2}(0)$ for all $\comment{t\geq t_1}$. Note that 
\begin{align*}
  \mynorm{e(t)}  \leq& \beta_2\mynorm{e_1}e^{-\beta_1(\alpha) (t-t_1)}+\frac{\beta_2}{\beta_1(\alpha)}(1-e^{-\beta_1(\alpha) (t-\tau)})\\
  &\sup_{t\geq t_0}( \delta_2 (t)+\gamma(t))\\
  \leq & \max \Big\{\beta_2\mynorm{e_1},\frac{\beta_2}{\beta_1(\alpha)}  \sup_{t\geq t_0}( \delta_2 (t)+\gamma(t)) \Big\}
\end{align*}
It can be verified that the condition ${e(t)}\in\mathcal{B}_{r_2}(0)$ will be satisfied if \eqref{eq: condition for tracking constrained} holds. \comment{Furthermore, by ${e(t)}\in\mathcal{B}_{r_2}(0)$ and Theorem \ref{lemma: exist and unique}, there must exist a unique solution for \eqref{eq: ODE3} for all $t\geq t_1$}.
\end{proof}
\begin{remark}
The inequality \eqref{eq: bound for tracking constrained} implies that the smaller the regularization parameter $\alpha$ is,  the smaller the tracking error $x(t,t_0,x_0)-h_2(t)$ is and the faster $x(t,t_0,x_0)$ converges to the neighbourhood of $h_2(t)$.
\end{remark}

\begin{remark}
In the case that the local minimum trajectory $h_2(t)$ is a constant, the upper bound on $\alpha$ simply becomes $\alpha<\infty$. This implies that if $h_2(t)$ is constant, then it will be perfectly tracked with any regularization parameter and can not be escaped by tuning the regularization parameter.
\end{remark}
\begin{remark}
In the unconstrained case or the case with the time-invariant constraints, $\delta_1(t)$ and $\delta_2(t)$ in \eqref{eq: perturbation 1} simply become zero. Then, the tracking conditions in \eqref{eq: condition for tracking constrained} become $\mynorm{x_0-h_2(0)}\leq r_2$ and $\alpha \leq \frac{c_2r_2}{\sup_{t\geq t_0} \gamma(t)}$, and the tracking error bound in \eqref{eq: bound for tracking constrained} becomes
\begin{align*}
 \nonumber \mynorm{e(t)}\leq& \mynorm{e_1}e^{-\frac{c_2}{\alpha} (t-t_1)}+\int_{t_1}^t e^{-\frac{c_2}{\alpha} (t-\tau)}\gamma(t)d\tau\\
 \leq & \frac{\alpha\sup_{t\geq \comment{t_1}}\gamma(t)}{c_2}
\end{align*}
\end{remark}
\comment{
\begin{remark}
After the solution of \eqref{eq: ODE3} has escaped the spurious local trajectories and started tracking the globally minimum trajectory, one may use the state-of-the-art tracking methods in \cite{tang2018running} and \cite{zavala2009line} to improve the tracking of the globally minimum trajectory.
\end{remark}}
\subsection{Escaping}
Combining the results of jumping and tracking immediately yields a sufficient condition on escaping from one local minimum trajectory to a more desirable local (or global) minimum trajectory. The proof is omitted for brevity.

\begin{theorem}[Sufficient conditions for escaping from $h_1(t)$ to $h_2(t)$] \label{thm: escape}
\comment{Given two local minimum trajectories $h_1(t)$ and $h_2(t)$, suppose that the Lagrange function $L(x,\lambda,t)$ with $\lambda$ given in \eqref{eq: lambda} is {locally $(c_2,r_2)$-one-point strongly convex}} with respect to $x$ around $h_2(t)$ in the time-varying feasible set $\{e\in \mathbb{R}^n: e+h_2(t)\in \mathcal{M}(t), \mynorm{e}\leq r_2\}$ and let $\mathcal{B}_v(h_1(t_1))\subseteq RA^{\mathcal{M}(t_1)}(h_1(t_1))$. Under the conditions of Theorem \ref{thm: jump} or \ref{thm: jump average}, if \eqref{eq: perturbation 1}-\eqref{eq: condition for tracking constrained} hold, then the solution of \eqref{eq: ODE3} will $(v,r_2)$-escape from $h_1(t)$ to $h_2(t)$ after $t\geq t_2$.
\end{theorem}

\subsection{Discussions}

\textbf{Adaptive inertia}: To leverage the potential of the time-varying perturbation 
$\alpha \mathcal{Q}(e(t)+h_2(t),t)g^\prime(e(t)+h_2(t),t) +\alpha \dot{h}_2(t)$
in re-shaping the landscape of the Langrange function or the objective function to become locally one-point strongly convex in $x$ over a large region, the regularization parameter $\alpha$ should be selected relatively large. On the other hand, to ensure that the solution of \eqref{eq: ode3 change variable} and \eqref{eq: ode2 change variable h2} will end up tracking a desirable local (or global) minimum trajectory, Theorem \ref{thm: Sufficient condition for tracking constrained} prescribes small values for  $\alpha$. In practice, especially when the time-varying objective function has many spurious shallow minimum trajectories, this suggests using a relatively large regularization parameter $\alpha$ at the beginning of the time horizon to escape  spurious shallow minimum trajectories and then switching to a relative small regularization parameter $\alpha$ for reducing the ultimate tracking error bound.
\vspace{2mm}

\noindent\textbf{Sequential jumping}: When the time-varying optimization problem has many local minimum trajectories, the solution of \eqref{eq: ODE3} or \eqref{eq: ODE2} may sequentially jump from one local minimum trajectory to a better local minimum trajectory. To illustrate this concept, consider the local minimum trajectories $h_1(t),h_2(t),...,h_{m}(t)$, where $h_{m}(t)$ is a global trajectory. Assume that there exists a sequence of time intervals $[t_1^i,t_2^i]$ for $i=1,2,\ldots,m-1$ such that  the conditions of Theorem \ref{thm: jump} or  \ref{thm: jump average} are satisfied for $h_i(t)$ and $h_{i+1}(t)$ during each time interval. Then, by sequentially deploying Theorem \ref{thm: jump} or  \ref{thm: jump average}, it can be concluded that the solution of \eqref{eq: ODE3} or \eqref{eq: ODE2} will jump from $h_1(t)$ to $h_{m}(t)$ after $t\geq t_2^m$. Furthermore, if $h_{m}(t)$ can be tracked with the given $\alpha$,  the solution of \eqref{eq: ODE3} or \eqref{eq: ODE2} will escape from $h_1(t)$ to $h_{m}(t)$ after $t\geq t_2^m$.

\section{Numerical Examples}\label{sec:numercial example}
\begin{figure}[ht]
\centering
\includegraphics[width=0.8 \linewidth]{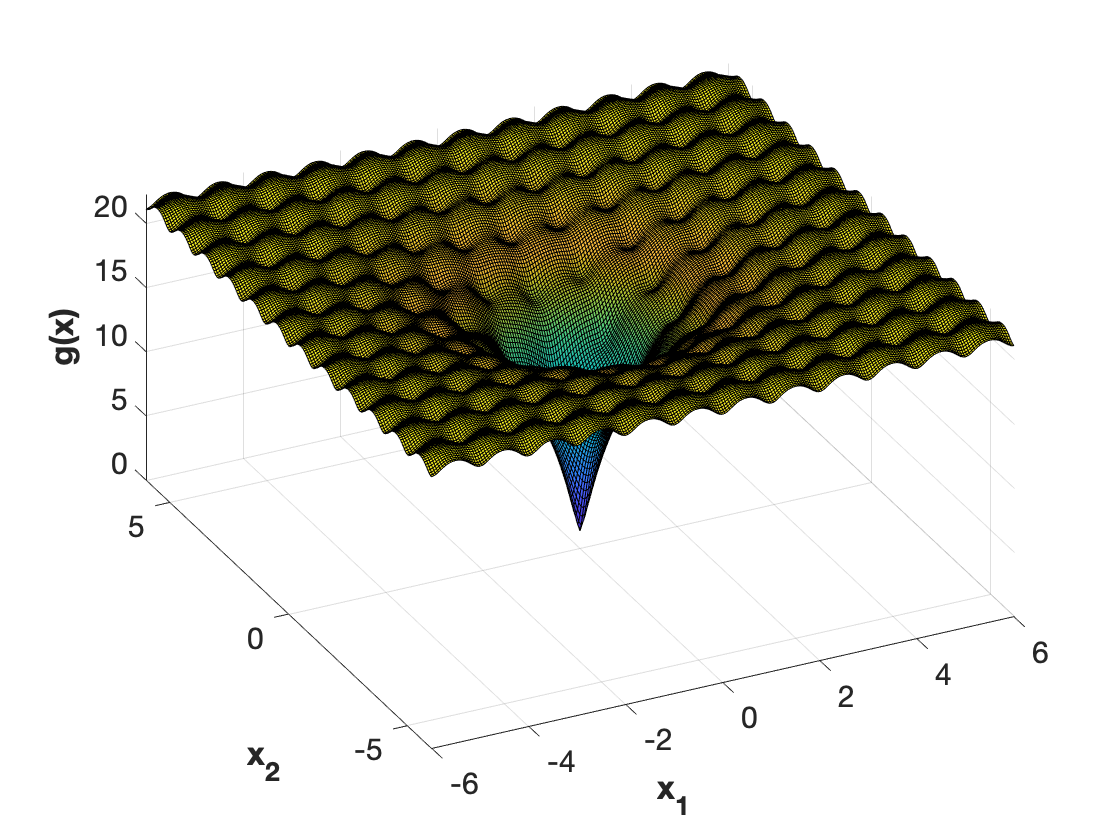}
\caption{Illustration of Example \ref{ex:ackley}.}
\label{fig:ackley_landscape}
\end{figure}
\begin{example} \label{ex:ackley} 
\emph{Consider the non-convex function}\end{example}
\vspace{-3mm}
\begin{align*}
  \bar{f}(x)=&0.5e+20e^{-d}-20e^{-\sqrt{0.5(x_1^2+x_2^2)+d^2}}\\
  &-0.5e^{(0.5(\cos(2\pi x_1)+\cos(2\pi x_2)))}.
\end{align*}
\comment{This function has a global minimum  at $(0,0)$ with the optimal value $0$ and many spurious local minima. Its landscape is shown in Figure \ref{fig:ackley_landscape}. When $d=0$, this function is called the Ackley function \cite{Ackley1987}, which is a benchmark function for  global optimization algorithms. To make this function twice continuously differentiable, we choose $d=0.01$. }

\comment{Consider the time-varying objective function $f(x,t)=\bar{f}(x-z(t))$ and the time-varying constraint $g(x,t)=(x_1-z_1(t))-1/2(x_2-z_2(t))^2=0$, where
$ z(t)=[24\sin(t), \cos(t)]^\top$.
This constrained time-varying optimization problem has the global minimum trajectory $[0,0]^\top+z(t)$ and many spurious local minimum trajectories. 
Two local minimum trajectories are $h_1(t)=[1.92,1.96]^\top+z(t)$ and $h_2(t)=[0,0]^\top+z(t)$. It can be shown that $L(x,\lambda,t)$ is locally $(20,0.5)$-one-point strongly convex with respect to $h_2(t)$. }

\comment{We take $D_{v,\rho,r_2}=D_{0.04,0.01,1}=[-0.1,2]\times[-0.1,2]$ in Definition \ref{def: dominant trajectory}. The condition in \eqref{eq: invariance-like set} can be verified by checking the signs of the derivatives of $e_1(t)$ and $e_2(t)$ along the dynamics \eqref{eq: ode3 change variable} on the boundary points of $D_{0.04,0.01,1}\cap \mathcal{M}^{h_2}(t)$.
Furthermore, \eqref{eq: strong convex for jump average} is satisfied for $w=1$. Thus, $h_2(t)$ is a $(0.2,1)$-dominant trajectory with respect to $h_1(t)$ during $[0,\frac{\pi}{8}]$ over the region $D_{0.04,0.01,1}$. }

\comment{Regarding Theorem \ref{thm: jump}, if we select $\theta=0.2$, the inequality \eqref{eq: t_2-t_1} is satisfied for $\alpha=0.2$ and $t_2-t_1=\pi/8$. Thus, the solution of \eqref{eq: ODE3} will $(0.04,0.5)$-jump from $h_1(t)$ to $h_2(t)$. Regarding Theorem \ref{thm: Sufficient condition for tracking constrained}, $\delta_1$ and $\delta_2$ in the inequality \eqref{eq: perturbation 1} can be taken as $0$ and $24\sqrt{2}\cos(t)+\sqrt{2}\sin(t)$, respectively. Then the inequality \eqref{eq: tracking upper bound on a} reduces to
$\alpha\leq \frac{10}{\sqrt{2(24^2+1)}}\approx 0.29$, which is satisfied by $\alpha=0.2$. Thus, the solution of \eqref{eq: ODE3} will $0.5$-track $h_2(t)$. Putting the above findings together, we can conclude that the solution of \eqref{eq: ode3 change variable} will $(0.04,0.5)$-escape from $h_1(t)$ to $h_2(t)$.}

\comment{In addition, by choosing the inertia parameter $\alpha=0.2$, the simulation shows that for 1000 runs of random initialization with $x_2(0)-z(0)\in[-5,5]$ and $x_1(0)$ determined by the equality constraint, all solutions of the corresponding \eqref{eq: ODE3} will sequentially jump over the local minimum trajectories and end up tracking the global trajectory after $t\geq 5\pi$. }

\section{Conclusion} \label{sec:conclusion}
In this work, we study the landscape of time-varying nonconvex optimization problems. The objective is to understand when simple local search algorithms can find (and track) time-varying global solutions of the problem over time. We introduce a time-varying projected gradient flow system with controllable inertia as a continuous-time limit of the optimality conditions for  discretized sequential optimization problems with proximal regularization and online updating scheme. 
Via a change of variables, the time-varying projected gradient flow system is regarded as a composition of a time-varying projected gradient term, a time-varying constraint-driven term and an inertia term due to the time variation of the local minimum trajectory. We show that the time-varying perturbation term due to the inertia encourages the exploration of the state space and re-shapes the landscape by potentially making it one-point strongly convex over a large region during some time interval. We introduce the notions of jumping and escaping, and use them to  develop sufficient conditions under which the time-varying solution escapes from a poor local trajectory to a better (or global) minimum trajectory over a finite time interval. We illustrate in a benchmark example with many shallow minimum trajectories that the natural time variation of the problem enables escaping  spurious local minima over time. 
Avenues for future work include the characterization of the class of problems in which all spurious local minimum trajectories are shallow compared with the global minimum trajectory. 


%

\appendices
\section{} \label{appendix: local lip of ODE3}
\begin{lemma} \label{lemma: local lip of ODE3}
Under Assumptions \ref{assumption:smoothness}-\ref{assumption: regular of constraints}, 
if $f(x,t)$, $g_1(x,t),\ldots,g_m(x,t)$ are twice continuously differentiable in $x$ on $D\times[a,b]$ for some domain $D\in \mathbb{R}^n$, then the function $\frac{1}{\alpha}\mathcal{P}(x,t)\nabla_x f(x,t)+\mathcal{Q}(x,t)g^\prime(x,t)$ is locally Lipschitz continuous in $x$ on $\{(x,t)\in D\times[a,b]:x\in \mathcal{M}(t)\}$.
\end{lemma}
\begin{proof}
We first show that the matrix $\mathcal{J}_g(x,t)\mathcal{J}_g(x,t)^\top$ is positive definite over the time-varying feasible region $\mathcal{M}(t)$ for all $t\in[a,b]$. Since, under Assumption \ref{assumption: regular of constraints}, $\mathcal{J}_g(x,t)$ is full row rank for $x\in \mathcal{M}(t)$ for all $t\in[a,b]$, the null space of $\mathcal{J}_g(x,t)$ is $0$. Thus, $\mathcal{J}_g(x,t) x=0$ if and only if $x=0$. Therefore, $x\mathcal{J}_g(x,t)\mathcal{J}_g(x,t)^\top x^\top=(x\mathcal{J}_g(x,t))(x\mathcal{J}_g(x,t))^\top >0$ for all $x\neq 0$. Therefore,  $\mathcal{J}_g(x,t)\mathcal{J}_g(x,t)^\top$ is positive definite. 
Denote $\sigma_{\min}(\mathcal{J}_g(x,t)\mathcal{J}_g(x,t)^\top)$ as the minimum eigenvalue of $\mathcal{J}_g(x,t)\mathcal{J}_g(x,t)^\top$. Then,there exists a positive constant $c$ such that $\sigma_{\min}(\mathcal{J}_g(x,t)\mathcal{J}_g(x,t)^\top)\geq c$.
By the chain rule and the twice continuously differentiability of $g_k(x,t)$'s in $x\in D$ for $t\in[a,b]$, we know that $G(x,t):=\mathcal{J}_g(x,t)\mathcal{J}_g(x,t)^\top$ is continuously differentiabile in $x\in D$ for $t\in[a,b]$. Next, we show that $G^{-1}$ is also continuously differentiabile in $x$ on $\{(x,t)\in D\times[a,b]:x\in \mathcal{M}(t)\}$.
Let $x_i$ be the $i$-th component of the vector $x$. By taking the derivative of the identity $I=G(x,t)G(x,t)^{-1}$ with respect to $x_i$, we obatin
\begin{equation*}
\frac{\partial}{\partial x_i}G(x,t)^{-1}=-   G(x,t)^{-1} \Big(\frac{\partial}{\partial x_i}G(x,t)\Big)G(x,t)^{-1}
\end{equation*}
and
\begin{align*}
\mynorm{\frac{\partial}{\partial x_i}G(x,t)^{-1}}&\leq \mynorm{G(x,t)^{-1}}^2 \mynorm{\frac{\partial}{\partial x_i}G(x,t)}\\
&\leq \frac{1}{c^2}\mynorm{\frac{\partial}{\partial x_i}G(x,t)}<\infty
\end{align*}
Since the inversion operator is continuous by the Cayley–Hamilton theorem, $G(x,t)$ is continuously differentiable in $x$ and $\mynorm{\frac{\partial}{\partial x_i}G(x,t)^{-1}}$ is bounded, we know that $\frac{\partial}{\partial x_i}G(x,t)^{-1}$ is well-defined and continuous over $D$ for all $i=1,\ldots,n$. Therefore, $G(x,t)^{-1}$ is continuously differentiable in $x$ on $\{(x,t)\in D\times[a,b]:x\in \mathcal{M}(t)\}$. Consequently,
 $\mathcal{P}(x,t)$ and $\mathcal{Q}(x,t)$ are continuously differentiable in $x$ on $\{(x,t)\in D\times[a,b]:x\in \mathcal{M}(t)\}$. Since $f(x,t)$ is twice continuously differentiable in $x$ and $g_k(x,t)$'s are twice continuously differentiable in $x\in D$ for $t\in[a,b]$, it holds that the function $\frac{1}{\alpha}\mathcal{P}(x,t)\nabla_x f(x,t)+\mathcal{Q}(x,t)g^\prime(x,t)$ is continuously differentiable in $x$. Hence, because of  \cite[Theorem 3.2]{khalil2002nonlinear}, it is locally Lipschitz continuous in $x$ on $\{(x,t)\in D\times[a,b]:x\in \mathcal{M}(t)\}$.
\end{proof}


\section*{Acknowledgment}
This work was supported by grants from ARO, AFOSR, ONR and NSF. 

\ifCLASSOPTIONcaptionsoff
  \newpage
\fi



%

\bibliographystyle{IEEEtran}
\bibliography{IEEEabrv,reference}

%
\vspace{-1cm}
\begin{IEEEbiography}[{\includegraphics[width=1in,height=1.25in,clip,keepaspectratio]{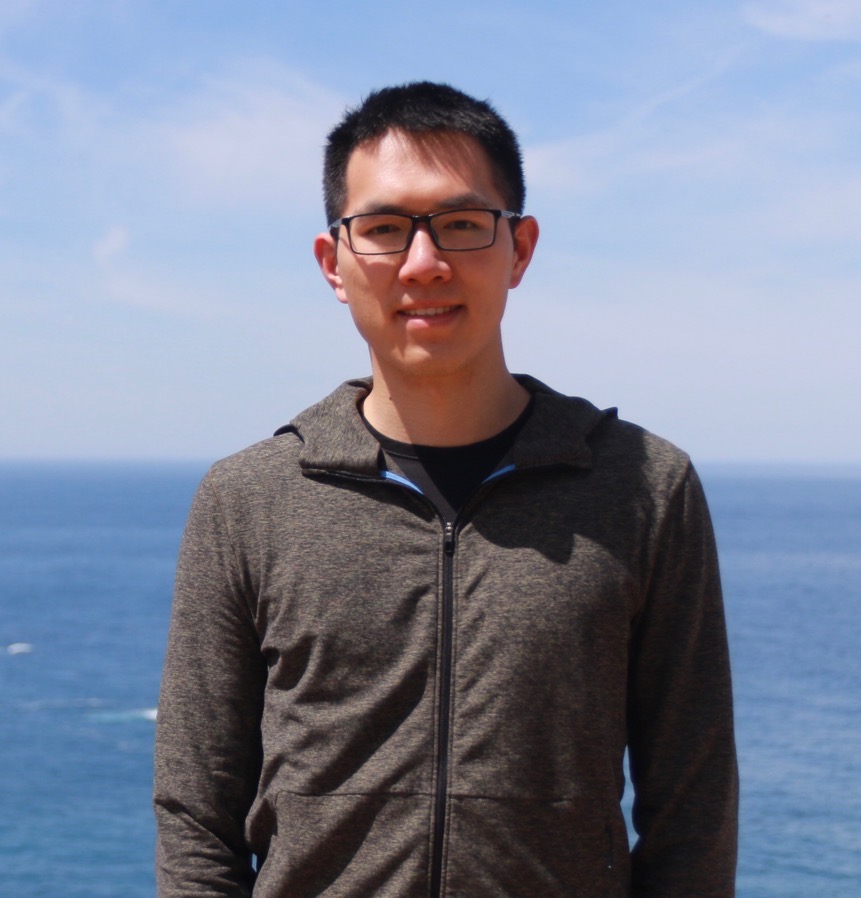}}]{Yuhao Ding} is currently working toward the Ph.D. degree in Industrial Engineering and Operations Research at
the University of California, Berkeley,
CA, USA. He obtained the B.E. degree in Aerospace Engineering from Nanjing University of Aeronautics and Astronautics in 2016, and the M.S. degree in Electrical and Computer Engineering from University of Michigan, Ann Arbor in 2018.
\end{IEEEbiography}
\vspace{-1cm}
\begin{IEEEbiography}[{\includegraphics[width=1in,height=1.25in,clip,keepaspectratio]{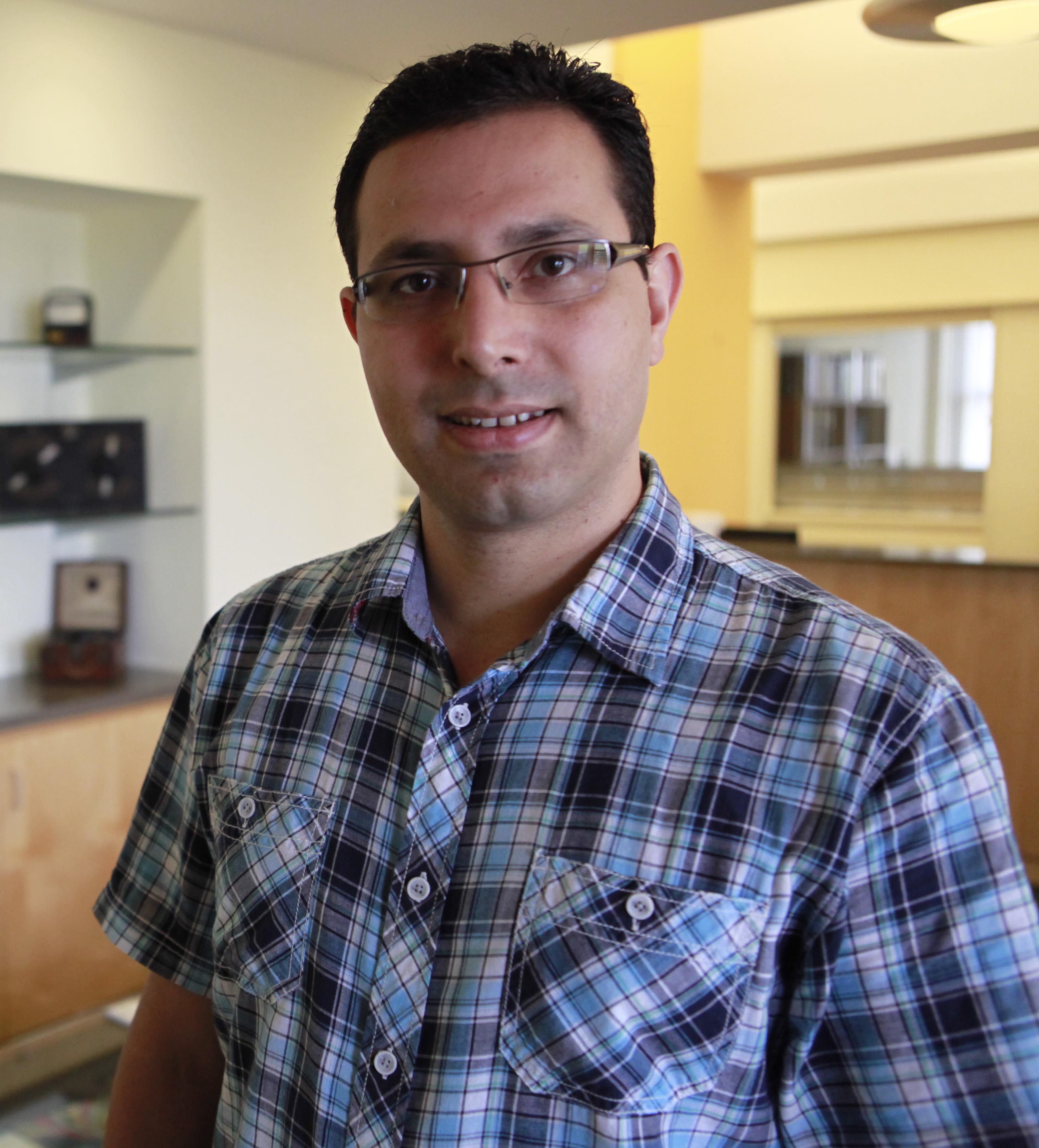}}]{Javad Lavaei} is currently an Associate Professor in the
Department of Industrial Engineering and Operations Research at the University of California, Berkeley,
CA, USA. He has
worked on different interdisciplinary problems in
power systems, optimization theory, control theory,
and data science. 
He is an associate editor of the IEEE Transactions on Automatic Control, the
IEEE Transactions on Smart Grid, and the IEEE Control System
Letters. He serves on the conference editorial boards of the IEEE Control
Systems Society and European Control Association.\end{IEEEbiography}
\vspace{-1cm}
\begin{IEEEbiography}[{\includegraphics[width=1in,height=1.25in,clip,keepaspectratio]{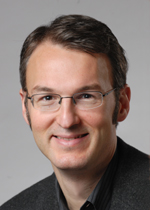}}]{Murat Arcak} is currently a Professor with the Department of Electrical Engineering and Computer
Sciences, University of California, Berkeley, CA,
USA. 
His research interests include dynamical systems and control theory with applications to
synthetic biology, multiagent systems, and transportation.
Prof. Arcak was a recipient of the CAREER Award from the National
Science Foundation in 2003, the Donald P. Eckman Award from the
American Automatic Control Council in 2006, the Control and Systems
Theory Prize from the Society for Industrial and Applied Mathematics
(SIAM) in 2007, and the Antonio Ruberti Young Researcher Prize from
the IEEE Control Systems Society in 2014. He is a member of SIAM.
\end{IEEEbiography}




\end{document}